\documentclass[12pt, a4paper]{amsart}
\usepackage{amsmath, amsfonts, amssymb, amsthm}
\usepackage[margin=1.1in, top=1in, bottom=1in]{geometry}
\usepackage{tikz}
\usepackage{enumerate}
\usepackage{dsfont}
\usepackage[T1]{fontenc}

\newcommand{\one}{\ensuremath{\mathds{1}}}
\newcommand{\C}{\mathbb{C}}

\newcommand{\N}{\mathbb{N}}
\newcommand{\R}{\mathbb{R}}
\newcommand{\T}{\mathbb{T}}
\newcommand{\Z}{\mathbb{Z}}

\newcommand{\Ind}{\operatorname{Ind}}
\newcommand{\I}{\mathcal{I}}
\newcommand{\F}{\mathcal{F}}

\newcommand{\Arg}{\operatorname{Arg}}
\newcommand{\End}{\operatorname{End}}
\newcommand{\Iso}{\operatorname{Iso}}
\newcommand{\Tr}{\operatorname{Tr}}
\newcommand{\Per}{\operatorname{Per\, \Lambda}}
\newcommand{\BB}{\mathcal{B}}
\newcommand{\GG}{\mathcal{G}}
\newcommand{\HH}{\mathcal{H}}
\newcommand{\GL}{\mathcal{G}_\Lambda}
\newcommand{\Gx}{\mathcal{G}^x_x}
\newcommand{\GU}{\mathcal{G}^{(0)}}
\newcommand{\CGx}{C^*(\mathcal{G}^x_x,\BB)}
\newcommand{\CG}{C^*(\mathcal{G},\BB)}
\newcommand{\PCGx}{\Gamma_c(\mathcal{G}^x_x;\BB)}
\newcommand{\PCG}{\Gamma_c(\mathcal{G};\BB)}
\newcommand{\Gzero}{\Gamma_0(\mathcal{G}^{(0)};\BB)}

\newcommand{\LL}{\mathcal{L}}
\newcommand{\PP}{\mathcal{P}}
\newcommand{\HL}{\mathcal{H}_\Lambda}

\newcommand{\FF}{\mathcal{F}}

\newcommand{\supp}{\operatorname{supp}}
\newcommand{\clsp}{\overline{\operatorname{span}}}

\allowdisplaybreaks

\newtheorem{thm}{Theorem}
\newtheorem{lemma}[thm]{Lemma}
\newtheorem{cor}[thm]{Corollary}

\theoremstyle{definition}
\newtheorem{definition}[thm]{Definition}

\newtheorem{remark}[thm]{Remark}

\numberwithin{equation}{section}
\numberwithin{thm}{section}

\title[KMS states]{\boldmath{KMS states  on the $C^*$-algebras of Fell Bundles over groupoids}}
\date{2 August 2017}
\author{Zahra Afsar}

\author{Aidan Sims}

\address{Zahra Afsar and Aidan Sims, School of Mathematics and Applied Statistics, University of Wollongong,
NSW 2522, Australia}

\email{\{zafsar, asims\}@uow.edu.au}

\thanks{This research was supported by the Australian
Research Council, grant number DP150101595.}

\begin{document}

\begin{abstract}
We consider fibrewise singly generated Fell-bundles over \'etale groupoids. Given a
continuous real-valued 1-cocycle on the groupoid, there is a natural dynamics on the
cross-sectional algebra of the Fell bundle. We study the Kubo--Martin--Schwinger
equilibrium states for this dynamics. Following work of Neshveyev on equilibrium states
on groupoid $C^*$-algebras, we describe the equilibrium states of the cross-sectional
algebra in terms of measurable fields of traces on the $C^*$-algebras of the restrictions
of the Fell bundle to the isotropy subgroups of the groupoid. As a special case, we
obtain a description of the trace space of the cross-sectional algebra. We apply our
result to generalise Neshveyev's main theorem to twisted groupoid $C^*$-algebras, and
then apply this to twisted $C^*$-algebras of strongly connected finite $k$-graphs.

\end{abstract}

\maketitle

\section{Introduction}

The study of KMS states of $C^*$-algebras was originally motivated by applications of
$C^*$-dynamical systems to the study of quantum statistical mechanics \cite{BR}. However,
KMS states make sense for any $C^*$-dynamical system, even if it does not model a
physical system, and there is significant evidence that the KMS data is a useful
invariant of a dynamical system. For example, the results of Enomoto, Fujii and Watatani
\cite{EFW} show that the KMS data for a Cuntz--Krieger algebra encodes the topological
entropy of the associated shift space. And Bost and Connes showed that the Riemann zeta
function can be recovered from the KMS states of an appropriate $C^*$-dynamical system
\cite{BC}. As a result there has recently been significant interest in the study of KMS
states of $C^*$-dynamical systems arising from combinatorial and algebraic data \cite{BC,
CarlsenLarsen, N, Thomsen, Kakariadis}. In particular, there are indications of a close
relationship between KMS structure of such systems, and ideal structure of the
$C^*$-algebra \cite{aHLRS, LLNSW, Yang}.

Our original motivation in this paper was to investigate whether the relationship,
discovered in \cite{aHLRS}, between simplicity and the presence of a unique state for the
$C^*$-algebra of a strongly connected $k$-graph persists in the situation of twisted
higher-rank graph $C^*$-algebras. The methods used to establish this in \cite{aHLRS}
exploit direct calculations with the generators of the $C^*$-algebra. Unfortunately, a
similar approach seems to be more or less impossible in the situation of twisted
$k$-graph $C^*$-algebras, because the twisting data quickly renders the calculations
required unmanageable.

Instead we base our approach on groupoid models for $k$-graph $C^*$-algebras and their
analogues. Building on ideas from \cite{KR}, Neshveyev proved in \cite{N} that the KMS
states of a groupoid $C^*$-algebra for a dynamics induced by a continuous real-valued
cocycle on the groupoid are parameterised by pairs consisting of a suitably invariant
measure $\mu$ on the unit space, and an equivalence class of $\mu$-measurable fields of
traces on the $C^*$-algebras of the fibres of the isotropy bundle that are equivariant
for the natural action of the groupoid by conjugation. Though Neshveyev's results are not
used directly to compute the KMS states of $k$-graph algebras in \cite{aHLRS}, it is
demonstrated in \cite[Section~12]{aHLRS} that the main results of that paper could be
recovered using Neshveyev's theorems.

Every twisted $k$-graph algebra can be realised as a twisted groupoid $C^*$-algebra
\cite{KPS1}, and simplicity of twisted $k$-graph algebras can be characterised using this
description \cite{KPS2}. Twisted $k$-graph $C^*$-algebras are in turn a special case of
cross-sectional algebras of Fell bundles over groupoids. Since the latter constitute a
very flexible and widely applicable model for $C^*$-algebraic representations of
dynamical systems, we begin by generalising Neshveyev's theorems to this setting; though
since it simplifies our results and since it covers our key example of twisted groupoid
$C^*$-algebras, we restrict to the situation of Fell bundles whose fibres are all singly
generated. Neshveyev's approach relies heavily on Renault's Disintegration Theorem
\cite{Re}, and we likewise rely very heavily on the generalisation of the Disintegration
Theorem to Fell-bundle $C^*$-algebras established by Muhly and Williams \cite{MW}.

Our first main theorem, Theorem~\ref{thm2}, is a direct analogue in the situation of Fell
bundles of Neshveyev's result. It shows that the KMS states on the cross-sectional
algebra of a Fell bundle $\BB$ with singly generated fibres over an \'etale groupoid
$\GG$ are parameterised by pairs consisting of a suitably invariant measure $\mu$ on
$\GU$ and a $\mu$-measureable field of traces on the $C^*$-algebras $C^*(\Gx, \BB)$ of
the restrictions of $\BB$ to the isotropy groups of $\GG$ that satisfies a suitable
$\GG$-invariance condition. By applying this result with inverse temperature equal to
zero, we obtain a description of the trace space of $C^*(\GG, \BB)$.

Given a continuous $\T$-valued $2$-cocycle $\sigma$ on $\GG$, or more generally a twist
over $\GG$ in the sense of Kumjian \cite{Kum}, there is a Fell line-bundle over $\GG$
whose cross-sectional algebra coincides with the twisted $C^*$-algebra $C^*(\GG,\sigma)$
(see Lemma~\ref{Fell-groupoid}). We apply Theorem~\ref{thm2} to such bundles to obtain a
generalisation of Neshveyev's results \cite[Theorem~1.2 and Theorem~1.3]{N} to twisted
groupoid $C^*$-algebras (see Corollary~\ref{kms states for groupoid}).

We next consider a strongly connected $k$-graph $\Lambda$ in the sense of \cite{KP}.
There is only one probability measure $M$ on the unit space $\GL^{(0)} = \Lambda^\infty$
that is invariant in the sense described above \cite[Lemma~12.1]{aHLRS}. Given a cocycle
$c$ on $\Lambda$, Kumjian, Pask and the second author introduced a twisted $C^*$-algebra
$C^*(\Lambda,c)$ and showed that the cocycle $c$ induces a cocycle $\sigma_c$ on the
associated path groupoid $\GL$ such that the $C^*$-algebras $C^*(\Lambda,c)$ and
$C^*(\GL,\sigma_c)$ are isomorphic \cite[Corollary~7.9]{KPS1}. The cocycle $\sigma_c$
determines an antisymmetric bicharacter $\omega_c$ on $\Per$ (see \cite{OPT} or
\cite[Proposition~3.1]{KPS2}). The trace simplex of $C^*(\Per, \sigma_c)$ is canonically
isomorphic to the state space of the commutative subalgebra $C^*(Z_{\omega_c})$ of the
centre of the bicharacter $\omega_c$ (see Lemma~\ref{lemma:traces}). Conjugation in the
line-bundle associated to $\sigma_c$ determines an action of the quotient $\HL$ of $\GL$
by the interior of its isotropy on $\Lambda^\infty \times Z_{\omega_c}$. Kumjian, Pask
and the second author showed that $C^*(\Lambda, c)$ is simple if and only if this action
is minimal. Here we prove that the KMS states of $C^*(\Lambda, c)$ are parameterised by
$M$-measurable fields of traces on $C^*(Z_{\omega_c})$ that are invariant for the same
action of $\HL$. Unfortunately, however, we have been unable to prove that minimality of
the action implies that it admits a unique invariant field of traces.

\medskip

We begin with a section on preliminaries. We show if $\omega$ is an antisymmetric
bicharacter on a finitely generated free abelian group $F$ that is cohomologous to a
cocycle  on $P$, then the trace spaces of $C^*(P,\omega)$ and $C^*(Z_\omega)$ are
isomorphic. In Section~\ref{sec:KMS-Fell}, we prove our main theorems about the KMS
states on the cross-sectional algebra of a Fell bundle. In Section~\ref{sec:KMS twisted
G}, we construct a Fell bundle from a cocycle on a groupoid, and use our results in
Section~\ref{sec:KMS-Fell} to obtain a twisted version of Neshveyev's results in
\cite{N}. Section~\ref{sec:KMS k-graph} contains our results about the preferred dynamics
on the twisted $C^*$-algebras of $k$-graphs. We finish off by posing the question whether simplicity
of $C^*(\Lambda, c)$ implies that it admits a unique KMS state.

\section{Preliminaries}
Throughout this paper $\T$ is regarded as a multiplicative group with identity $1$.
\subsection{Groupoids}
Let $\GG$ be a locally compact second countable Hausdorff groupoid (see \cite{Re}). For
each $x\in \GU$, we write $\GG^x=r^{-1}(x)$, $\GG_x=s^{-1}(x)$ and
$\GG^x_x=\GG_x\cap\GG^x$. The set $\Iso(\GG):=\bigcup_{x\in \GU}\GG^x_x$ is called the
\textit{isotropy} of $\GG$. We say $\GG$ is \'{e}tale if $r$ and $s$ are local
homeomorphisms.  A bisection of $\GG$ is an open subset $U$ of $\GG$ such that $r|_{U}$
and $s|_{U}$ are homeomorphisms.

A \textit{continuous $\T$-valued 2-cocycle} $\sigma$ on $\GG$ is a continuous function
$\sigma:\GG^2\rightarrow \T$ such that $\sigma(r(\gamma),\gamma)=c(\gamma,s(\gamma))=1$
for all $\gamma\in \GG$ and
$\sigma(\alpha,\beta)\sigma(\alpha\beta,\gamma)=\sigma(\beta,\gamma)\sigma(\alpha,\beta\gamma)$
for all composable  triples $(\alpha,\beta,\gamma)$. We write $Z^2(\GG,\T)$ for the group
of all continuous $\T$-valued 2-cocycles on $\GG$. Let $b:\GG\rightarrow \T$ be a
continuous function such that $b(x)=1$ for all $x\in \GU$. The function  $\delta^1 b:
\GG\times \GG\rightarrow \T$ given by $\delta^1
b(\gamma,\alpha)=b(\gamma)b(\alpha)\overline{b(\gamma\alpha)}$ is a continuous  2-cocycle
and is  called the \textit{$2$-coboundary} associated to $b$. Note that if $b$ is
continuous ,then $\delta^1 b$ is a $\T$-valued 2-cocycle on $\GG$. Two continuous
$\T$-valued 2-cocycles $\sigma,\sigma'$ are \textit{cohomologous} if
$\sigma'\overline{\sigma}=\delta^1b$ for some continuous $b$. A \textit{continuous
$\R$-valued 1-cocycle} $D$ on $\GG$ is   a continuous homomorphism from $D$ to $\R$.

Given $\sigma\in Z^2(\GG,\T)$, the space $C_c(\GG)$  is a $*$-algebra with the involution
and multiplication defined  by
\[
f^*(\gamma):=\overline{\sigma(\gamma,\gamma^{-1})f(\gamma^{-1})} \text{  and }
\]
\[
(fg)(\gamma):=\sum_{\alpha\beta=\gamma}\sigma(\alpha,\beta)f(\alpha)g(\beta)\quad \text{for } f,g\in C_c(\GG).
\]
We denote this $*$-algebra by $C_c(\GG,\sigma)$. The formula
\[\|f\|_I=\max\Big(\sup_{x\in \GU}\sum_{\lambda\in \GG^x}|f(\lambda)|,\sup_{x\in \GU}\sum_{\lambda\in \GG_x}|f(\lambda)| \Big)\]
determines a norm on $C_c(\GG,\sigma)$. By a  $*$-representation  of $C_c(\GG,\sigma)$,
we mean a $*$-homomorphism from $C_c(\GG,\sigma)$ to the bounded operators on a Hilbert
space. The \textit{twisted groupoid} $C^*$-algebra $C^*(\GG,\sigma)$ is the completion of
$C_c(\GG,\sigma)$ in the  universal norm
\[\|f\|:=\sup\{\|L(f)\|: L \text{  is a  $*$-representation of } C_c(\GG,\sigma)\}. \]

A measure $\mu$ on $\GU$ is called \textit{quasi-invariant} if the  measures
\[\nu(f):=\int_{\GU}\sum_{\gamma\in\GG^x}f(\gamma)\,d\mu \quad\text{ and  }\quad \nu^{-1}(f):=\int_{\GU}\sum_{\gamma\in\GG_x}f(\gamma)\,d\mu\]
are equivalent. We write $\Delta_\mu=\frac{d\nu}{d\nu^{-1}}$ for a Radon--Nykodym
derivative of $\nu$ with respect to $\nu^{-1}$. We will call $\Delta_\mu$ the
\textit{Radon--Nykodym cocycle} of $\mu$. Given a bisection $U$ and $x\in \GG^{(0)}$, let
$U^x:=U\cap r^{-1}(x)$. Define $T_U:r(U)\rightarrow s(U)$ by $T(x)=s(U^x)$. To see that a
measure $\mu$ is quasi-invariant it suffices to show that
\[
\int_{r(U)} f(T_U(x))\, d\mu(x)=\int_{s(U)} f(x)\Delta_\mu(U_x)\, d\mu(x)
\]
for all bisections $U$ and all  $f:s(U)\rightarrow \R$.

\subsection{Fell bundles}
Let $C,D$ be $C^*$-algebras.  A $C$--$D$ bimodule $Y$ is said to be a
$C$--$D$-imprimitivity bimodule if it is  a full left Hilbert  $C$-module and a full
right Hilbert  $D$-module; and   for all $y,y',y''\in Y$, $c\in C$ and $d\in D$, we have
\begin{align}\label{imprimitivity}
_C\langle  y\cdot d,y'\rangle=_C\langle y,y'\cdot d^*\rangle,\quad\notag& \langle c\cdot y,y'\rangle_D=\langle y,c^*\cdot y'\rangle_D\text{ and }\\
_C\langle  y,y'\rangle\cdot y''&=y\cdot \langle y',y''\rangle_D.
\end{align}
Let $\GG$ be a locally compact second countable \'{e}tale groupoid. Suppose that
$p:\BB\rightarrow \GG$ is a separable upper-semi continuous Banach bundle over $\GG$ (see
\cite[Definition~A.1]{MW}).  Let
\[\BB^2:=\{(a,b)\in \BB\times\BB:(p(a),p(b))\in \GG^2\}.\]
Following \cite{MW}, we say $\BB$ is a \textit{ Fell bundle} over $\GG$ if there is a
continuous involution $a\mapsto a^*:\BB\rightarrow \BB$ and a continuous bilinear
associative multiplication $(a,b)\mapsto ab:\BB^2\rightarrow \BB$ such that
\begin{itemize}
\item [(F1)] $p(ab)=p(a)p(b)$,
\item [(F2)]$p(a^*)=p(a)^{-1}$,
\item [(F3)]$(ab)^*=b^*a^*$,
\item [(F4)]for each $x\in \GU$, the fibre $B(x)$ is a $C^*$-algebra with respect to
    the $*$-algebra structure given by the above involution and multiplication, and
\item [(F5)] for each $\gamma\in \GG$, $B(\gamma)$ is a
    $B(r(\gamma))$-$B(s(\gamma))$-imprimitivity bimodule with actions induced by the
    multiplication and the inner products
\begin{equation}\label{f5}
_{B(r(\gamma))}\langle a,b\rangle=ab^* \text { and }  \langle a,b\rangle_{B(s(\gamma))}=a^*b.
\end{equation}
\end{itemize}
For $x\in \GU,$ we sometimes write $A(x)$ for the fibre $B(x)$ to emphasis on its
$C^*$-algebraic structure. Given a Fell bundle $\BB$ over $\GG$, we say the fibre
$B(\gamma)$ is \textit{singly generated} if there exists an element $\one_\gamma\in
B(\gamma)$ such that
\[_{A(r(\gamma))}\langle \one_\gamma,\one_\gamma\rangle=\one_\gamma\one_\gamma^*=1_{A(r(\gamma))}, \quad \langle \one_\gamma,\one_\gamma\rangle_{A(s(\gamma))}=\one_\gamma^*\one_\gamma=1_{A(s(\gamma))} \text { and}\]
\[
B(\gamma)=A(r(\gamma)) \one_\gamma=\one_\gamma A(s(\gamma)).
\]
In particular, for $x\in \GU$, the fibre $A(x)$ is singly generated if and only if it is
a unital $C^*$-algebra, and we can then take
 $\one_x=1_{A(x)}$.

A continuous function $f:\GG\rightarrow \BB$ is a \textit{section} if $p\circ f$ is the
identity map on $\GG$. A section $f$ \textit{vanishes at infinity} if the set
$\{\gamma\in \GG:\|f(x)\|\geq \epsilon\}$ is compact for all $\epsilon>0$. We write
$\Gamma_0(\GG;\BB)$ for the completion of the set of sections which vanishes at infinity
with respect to the norm $\|f\|:=\sup_{\gamma\in \GG}\|f(\gamma)\|$.  The space
$\Gamma_0(\GG;\BB)$ is a Banach space, see for example \cite[Proposition~C.23]{W}.

A Fell bundle $\BB$ over $\GG$ has \textit{enough sections} if for every $\gamma\in \GG$
and $a\in \BB(\gamma)$, there is a section $f$ such that
 $f(\gamma)=a$. If $\GG$ is a locally compact Hausdorff space, then   $p:\BB\rightarrow \GG$ has  enough sections, see \cite[Appendix~C]{FD}.

The space $\PCG$ of compactly supported continuous sections is a $*$-algebra with
involution and multiplication given by
\begin{gather}
    f^*(\gamma):=f(\gamma^{-1})^* \text{  and }\label{invo-formula}\\
    f*g(\gamma):=\sum_{\alpha\beta=\gamma}f(\alpha)g(\beta)\quad \text{for } f,g\in \PCG.\label{multi-formula}
\end{gather}
The $I$-norm on $\PCG$ is given by
\[\|f\|_I=\max\Big(\sup_{x\in \GU}\sum_{\lambda\in \GG^x}\|f(\lambda)\|,\sup_{x\in \GU}\sum_{\lambda\in \GG_x}\|f(\lambda)\| \Big).\]
A $*$-homomorphism $L:\PCG\rightarrow B(\HH_L)$ is \textit{$I$-norm decreasing
representation} if $\clsp\{L(f)\xi:f\in \PCG,\xi\in \HH_L=\HH_L\}$ and if $\|L(f)\|\leq
\|f\|_I$ for all $f\in \PCG$. The \textit{universal $C^*$-norm } on $\PCG$ is
\[\|f\|:=\sup\{\|L(f)\|: L \text{  is a  $I$-norm decreasing representation}\}, \]
and $\CG$ is the completion of  $\PCG$ with respect to the universal norm.

Let $\FF$ be a closed subgroupoid of $\GG$.  Then $\BB|_\FF$ is a Fell bundle over $\FF$.
We write $\Gamma_c(\FF;\BB)$ in place of $\Gamma_c(\FF;\BB|_\FF)$ and we denote the
completion $\Gamma_c(\FF;\BB)$ in the universal norm by $C^*(\FF,\GG)$.

Suppose that each fibre in $\BB$ is singly generated. Fix $x\in \GU$. For $u\in \Gx$ and
$a\in B(u)$, let $a\cdot \delta_u\in \Gamma_c(\Gx;\BB)$ be the section given by
\begin{align*}
a\cdot\delta_u(v)=\begin{cases}
a&\text{if $u=v$}\\
	0&\text{otherwise.}\\
	\end{cases}
\end{align*}
Then
\[\CGx=\clsp\{a\cdot\delta_u:u\in \Gx, a\in B(u)\}.\]
In particular $\CGx$ is a unital $C^*$-algebra  with $1_{\CGx}=\one_x\cdot \delta_x.$

\subsection{Representations of Fell bundles and the Disintegration Theorem}
Let $p:\BB\rightarrow \GG$ be a Fell bundle over a locally compact second countable
\'{e}tale groupoid $\GG$. Suppose that $\GU*\HH$ is a Borel Hilbert bundle over $\GU$ as
in \cite[Definition~F.1]{W}.  Let
\[\End(\GU*\HH):=\{(x,T,y): x,y\in \GU, T\in B\big( \HH(y),\HH(x)\big)\}.\]
Following \cite[Definition~4.5]{MW}, we say a map $\hat{\pi}:\BB\rightarrow
\End(\GU*\HH)$ is a \textit{$*$-functor} if each $\hat{\pi}(a)$ has the form
$\hat{\pi}(a)=(r(p(a)),\pi(a),s(p(a)))$ for some
 $\pi(a):\HH(s(p(a)))\rightarrow \HH(r(p(a)))$ such that the maps $\pi(a)$ collectively satisfy
\begin{itemize}
\item [(S1)] $\pi(\lambda a+b)=\lambda\pi(a)+\pi(b)$ if $p(a)=p(b)$,
\item [(S2)] $\pi(ab)=\pi(b)\pi(a)$ whenever $(a,b)\in \BB^2$, and
\item [(S3)]$\pi(a^*)=\pi(a)^*$.
\end{itemize}

A \textit{strict representation} of $\BB$ is a triple $(\mu, \GU*\HH,\hat{\pi})$
consisting of a quasi-invariant measure $\mu$ on $\GU$,  a Borel Hilbert bundle $\GU*\HH$
and a $*$-functor $\hat{\pi}$. For such a triple, we write $L^2(\GU*\HH,\mu)$ for the
completion of the set of all Borel sections $f:\GU\rightarrow \GU*\HH$ with
$\int_{\GU}\langle f(x), f(x)\rangle_{\HH(x)}\,d\mu(x)<\infty$ with respect to
\[\langle f, g\rangle_{L^2(\GU*\HH,\mu)}=\int_{\GU}\langle f(x), g(x)\rangle_{\HH(x)}\,d\mu(x).\]

Let $\Delta_\mu(u)$ be the Radon--Nikodym cocycle for $\mu$. Given a strict
representation $(\mu, \GU*\HH,\hat{\pi})$, Proposition~4.10 \cite{MW} gives an $I$-norm
bounded $*$-homomorphism $L$ on $L^2(\GU*\HH,\mu)$ such that
\begin{equation}\label{integratedform}
\big( L(f)\xi\big|\eta\big)=\int_{\GU}\sum_{u\in\GG^x}\big(\pi(f(u))\xi(s(u))\,\big|\,\eta(r(u))\big) \Delta_\mu(u)^{-\frac{1}{2}}\, d\mu(x).
\end{equation}
 We call $L$ \textit{the integrated form} of $\pi$.
The Disintegration Theorem \cite[Theorem~4.13]{MW} shows that  every non degenerate
representation  $M$   of $\CG$ is equivalent to the integrated form of a strict
representation.

\subsection{Cocycles and bicharacters on groups}\label{sec:bicharacters}

Let $F$ be a group. Viewing $F$ as a groupoid with the discrete topology, we write
$Z^2(F,\T)$ for the set of (continuous) $\T$-valued 2-cocycles on $F$. Given $\sigma\in
Z^2(F,\T)$, define $\sigma^*(p,q)=\overline{\sigma(q,p)}$. Proposition~3.2 of \cite{OPT}
implies that $\sigma,\sigma'\in Z^2(F,\T)$ are cohomologous if and only if
$\sigma\sigma^*=\sigma'\sigma'^*$.

Given  $\sigma\in Z^2(F,\T)$, the  $C^*$-algebra $C^*(F,\sigma)$ is the universal
$C^*$-algebra generated by unitaries $\{W_p :p\in F\}$ satisfying
$W_pW_q=\sigma(p,q)W_{pq}$ for all $p,q\in F$. A standard argument shows that if $\sigma$
and $\sigma'$ are cohomologous in $Z^2(F,\T)$, say $\sigma=\delta^1b \sigma'$,  then the
map $W_p\mapsto b(b)W_p$ descends to an isomorphism from $C^*(F,\sigma)$ onto
$C^*(F,\sigma')$, see for example \cite[Proposition3.5]{SWW}.

A \textit{bicharacter} on $F$ is a  function $\omega:F\times F\rightarrow \T$ such that
the functions $\omega(\cdot,p)$ and $\omega(q,\cdot)$ are homomorphisms.  A bicharacter
$\omega$ is \textit{antisymmetric} if $\omega(p,q)=\overline{\omega(q,p)}$. Each
bicharacter is a $\T$-valued 2-cocycle. If $F$ is a free abelian finitely generated
group, then \cite[Proposition~3.2]{OPT} shows that every $\T$-valued 2-cocycle $\sigma$
on $F$ is cohomologous to a  bicharacter: Let $q_1,\dots ,q_t$ be the generators of $F$.
Define a bicharacter $\omega:F\times F\rightarrow \T$ on  generators by
\begin{equation}\label{omega-formula}
\omega(q_i,q_j)=\begin{cases}
\sigma(q_i,q_j)\overline{\sigma(q_j,q_i)}&\text{if $i>j$.}\\
	1&\text{if $i\leq j$}\\
	\end{cases}
\end{equation}
Then $\omega\omega^*=\sigma\sigma^*$ and by  \cite[Proposition~3.2]{OPT}, $\omega$ is
cohomologous to $\sigma$.

Given $\sigma\in Z^2(F,\T)$, the map $p\mapsto (\sigma\sigma^*)(p,\cdot)$  is a
homomorphism from $F$ into the character space of $F$. Let
\[
Z_\sigma:=\{p\in F: \sigma\sigma^*(p,q)=1 \text{ for all } q\in F\}
\]
be  the kernel of the this  homomorphism. Therefore $Z_\sigma$  is a subgroup of $F$. If
$\omega$ is a bicharacter cohomologous to $\sigma$, then $Z_\omega=Z_\sigma$.

\begin{lemma}\label{lemma:traces}
Suppose that $F$ is a finitely generated free abelian group. Let $\sigma\in
Z^2(F,\mathbb{T})$ and let    $\omega$ be the  bicharacter defined in
\eqref{omega-formula}.  Then
\[\Tr(C^*(F,\sigma))\cong\Tr(C^*(F,\omega))\cong \Tr(C^*(Z_\omega))\cong \Tr(C^*(Z_\sigma)).\]
\end{lemma}
\begin{proof}
The first and third isomorphisms are clear. So we prove  the second isomorphism. We first
claim that for every  $\psi\in \Tr(C^*(F,\omega))$, we have
\[
\psi(W_p)=0 \text{ for all } p\notin Z_{\omega}.
\]
To see this, fix $p\notin Z_{\omega}$. There exists at least one generator $q_i\in F$
such that $(\omega\omega^*)(p,q_i)\neq 1$. Since $\psi$ is a trace and $\omega$ is a
bicharacter, we have
\begin{align*}
\psi(W_p)=\psi(W_{q_i}^*W_pW_{q_i})&=\omega(p,q_i)\omega({q_i}^{-1},pq_i)\psi(W_p)\\
&=\omega(p,q_i)\omega(q_i^{-1},p)\omega(q_i^{-1},q_i)\psi(W_p)\\
&=\omega(p,q_i)\overline{\omega(q_i,p)}\omega(q_i^{-1},q_i)\psi(W_p)\\
&=(\omega\omega^*)(q_i,p)\omega(q_i^{-1},q_i)\psi(W_p).
\end{align*}
The formula \eqref{omega-formula} for $\omega$ says that  $\omega(q_i^{-1},q_i)=1$. Since
$(\omega\omega^*)(q_i,p)\neq 1$, the above computation shows that  $\psi(W_p)=0$.

Next define $\Upsilon:C^*(F,\omega)\rightarrow C^*(Z_{\omega})$ on generators by
\begin{align*}
\Upsilon(W_p)=\begin{cases}
W_p&\text{if $p\in Z_{\omega}$}\\
	0&\text{if $p\notin Z_{\omega}$.}\\
	\end{cases}
\end{align*}
This induces a map $\Phi:\Tr( C^*(Z_{\omega}))\rightarrow \Tr(C^*(F,\omega))$ by
$\Phi(\psi)= \psi\circ\Upsilon$. The map $\Phi$ is clearly a continuius and affine map. The
embedding $\iota:C^*(Z_{\omega})\rightarrow C^*(F,\omega)$ induces a map
$\tilde{\iota}:\Tr( C^*(F,\omega))\rightarrow \Tr(C^*(Z_{\omega}))$ with
$\tilde{\iota}(\psi)= \psi\circ\iota$. A quick computation shows that $\tilde{\iota}$ and
$\Phi$ are inverses  of each other and therefore $\Phi$ is an isomorphism.
\end{proof}

\subsection{KMS states}
Let $\tau$ be an action of $\R$ by the automorphisms of a  $C^*$-algebra $A$. We say an
element $a\in A$ is \emph{analytic} if the map $t\mapsto \alpha_t(a)$  is the restriction
of an analytic function $z\mapsto\alpha_z(a)$ on $\C$. Following \cite{BR,aHLRS,N}, for
$\beta\in \R$, we say that a state $\psi$ of $A$ is a \emph{KMS$_\beta$} state (or KMS
state at inverse temperature $\beta$)  if $\psi(ab)=\psi(b\alpha_{i\beta}(a))$ for all
analytic elements $a,b$. It suffices  to check this condition (the KMS condition)  on a
set of analytic elements that span a dense subalgebra of $A$.   By \cite[Propositions~5.3.3]{BR}, all
KMS$_\beta$ states for $\beta \not= 0$ are $\tau$-invariant in the sense that
$\psi(\tau_t(a))=\psi(a)$ for all $t\in \R$ and $a\in A$.

\section{KMS states on the $C^*$-algebras of Fell bundles}\label{sec:KMS-Fell}

In \cite[Theorem~1.1 and Theorem~1.3]{N}, Neshveyev  described the KMS states of
$C^*$-algebras of locally compact second countable \'{e}tale groupoids.  Here, we
generalise his results to  the $C^*$-algebras of Fell bundles over  groupoids. Our proof
follows Neshveyev's closely.

Let $\mu$ be a probability measure on $\GU$. A \textit{$\mu$-measurable field of states}
is a collection $\{\psi_x\}_{x\in \GU}$ of states $\psi_x$ on $\CGx$ such that for every
$f\in \PCG$ the function $x\mapsto \sum_{u\in \Gx}\psi_x(f(u)\cdot\delta):\GU\rightarrow
\C$ is $\mu$-measurable. Given a $\mu$-measurable field $\Psi:=\{\psi_x\}_{x\in \GU}$ of states
we define
\begin{align*}
[\Psi]_\mu=\big\{\varphi: \varphi \text{ is a } \mu\text{-measurable field of states and } \varphi_x=\psi_x \text{ for  }\mu\text{-a.e. }x\in \GU\big\}.
\end{align*}
Given a state $\psi$  on a $C^*$-algebra $A$,  the \textit{centraliser}  of $\psi$ is the
set of all elements $a\in A$ such that
\[\psi(ab)=\psi(ba) \text { for all } b\in A.\]

\begin{thm}\label{thm1}
Let $p:\BB\rightarrow \GG$ be a Fell bundle with singly generated fibres over a locally
compact second countable \'{e}tale groupoid $\GG$. Let $\mu$ be a probability measure on
$\GU$ and let $\Psi := \{\psi_x\}_{x\in \GU}$ be a $\mu$-measurable field of tracial states. There
is a state $\Theta(\mu, \Psi)$ of $\CG$ with centraliser containing $\Gzero$ such that,
for $f \in \PCG$, we have
\begin{equation}\label{Nesh-formula}
\Theta(\mu,\Psi)(f)= \int_{\GG^{(0)}}\psi_x\big(f|_{\Gx}\big)\,d\mu(x)=\int_{\GG^{(0)}}\sum_{u\in \Gx}\psi_x\big(f(u)\cdot\delta_u\big)\, d\mu(x).
\end{equation}
We have $\Theta(\mu,\Psi) = \Theta(\nu, \Phi)$ if and only if $\mu = \nu$ and $[\Psi]_\mu
= [\Phi]_\mu$.
\end{thm}

We start with injectivity of the map induced by $\Theta$.

\begin{lemma}\label{injectivity}
Let $p:\BB\rightarrow \GG$ be a Fell bundle with singly generated fibres over a locally
compact second countable \'{e}tale groupoid $\GG$. If $\mu$ is a probability measure on
$\GU$ and $\Psi := \{\psi_x\}_{x \in \GU}$ and $\Psi' := \{\psi'_x\}_{x \in \GU}$ are
$\mu$-measurable fields of tracial states such that $\psi_x = \psi'_x$ for $\mu$-almost every 
$x$, then the functions $\Theta(\mu,\Psi)$ and $\Theta(\mu,\Psi')$ given
by~\eqref{Nesh-formula} agree. If $\psi$ is a state of $\CG$ with centraliser containing
$\Gzero$, then there is at most one pair $\big(\mu,[\Psi]_{\mu}\big)$ consisting of a
probability measure $\mu$ on $\GU$ and a $\mu$-equivalence class $[\Psi]_\mu$ of
$\mu$-measurable fields of tracial states on $\CGx$ such that that $\Theta(\mu,\Psi) =
\psi$.
\end{lemma}
\begin{proof}
The first statement is immediate from the definition of $\mu$-equivalence.

Now fix a state $\psi$ of $\CG$ with centraliser containing $\Gzero$. Suppose that
$\mu,\mu'$ are probability measures on $\GU$ and that $\Psi = \{\psi_x\}_{x\in \GU}$ and $\Psi' =
\{\psi'_x\}_{x\in \GU}$ are $\mu$-measurable fields of states satisfying $\Theta(\mu,\Psi) = \psi =
\Theta(\mu', \Psi')$. For each $f\in C_0(\GU)$, there is a section $\tilde{f}\in
\Gamma_c(\GG,\BB)\subseteq C^*(\GG,\BB)$ such that
\begin{align*}
\tilde{f}(\gamma)=
\begin{cases}
f(x)\one_x&\text{if } \gamma=x\in \GU\\
	0&\text{if } \gamma\notin \GU.\\
\end{cases}
\end{align*}
So~\eqref{Nesh-formula}, shows that
\[
\int_{\GG^{(0)}}\psi_x(\tilde{f}(x))\,d\mu(x)
    = \psi(\tilde{f})
    = \int_{\GG^{(0)}}\psi'_x(\tilde{f}(x))\,d\mu'(x).
\]
Since each $\tilde{f}(x) = f(x)\one_x$, and since each $\psi_x$ and each $\psi'_x$ is a
tracial state, we have $\psi_x(\tilde{f}(x)) = f(x) = \psi'_x(\tilde{f}(x))$ for all $x$,
and so $\int_{\GG^{(0)}} f \,d\mu = \psi(\tilde{f}) = \int_{\GG^{(0)}} f \,d\mu'$. So the
Riesz Representation Theorem shows that $\mu = \mu'$.

To see that $\psi$ and $\psi'$ agree $\mu$-almost everywhere, we suppose to the contrary
that $\psi_x\neq \psi'_x$ for some set $V\subseteq \GU$ with $\mu(V)\neq 0$ and derive a
contradiction. Since $\BB$ has enough sections, there is  a countable family $\F\subseteq
\Gamma_c\big(\GG\cup\cap\Iso(\GG);\BB\big)$ such that for each $\gamma\in
\GG\cup\cap\Iso(\GG)$, we have $\clsp\{f(\gamma):f\in \F\}=\BB(\gamma)$. So there is at
least one $f \in \F$ and $V' \subseteq V$ of nonzero measure, so that
\[
\psi\big(f|_{\Gx}\big)=\sum_{u \in G^x_x} \psi_x(f(u)\cdot\delta_u)
    \not= \sum_{u \in G^x_x} \psi_x(f(u)\cdot\delta_u)=\psi\big(f|_{\Gx}\big)\quad \text{for all }x\in V'.
\]
For each $l\in \N,$ let $V'_l:=\big\{x\in
V':\big|\psi_x\big(f|_{\Gx}\big)-\psi'_x\big(f|_{\Gx}\big)\big|> \frac{1}{l}\big\}$. So
there is  $l\in \N$ such that  $\mu( V'_l)>0$. Now for $0\leq j\leq 3$, let
\[V'_{l,j}:=\Big\{x\in V'_l: \Arg\big(\psi_x\big(f|_{\Gx}\big)-\psi'_x\big(f|_{\Gx}\big)\big)\in\Big[j\frac{\pi}{4},(j+1) \frac{\pi}{4}\Big]\Big\}.\]
 Therefore there is $j$ such that $\mu(V'_{l,j})>0$. Then
\[\Re \bigg(e^{-(\frac{\pi}{4}, i\frac{\pi}{2})}\int_{V_{l,j}} \Big(\psi_x\big(f|_{\Gx}\big)-\psi'_x\big(f|_{\Gx}\big)\Big)\,d\mu(x)\bigg)\geq \mu(V'_{l,j})\frac{1}{l\sqrt{2}}>0,\]
which is a contradiction.
\end{proof}

\begin{proof}[Proof of Theorem\ref{thm1}]
Fix a state $\psi$ of $\CG$ whose centraliser contains $\Gzero$. Let $(H,L,\xi)$ be the
corresponding GNS-triple. Applying the Disintegration Theorem (see \cite[Theorem
~4.13]{MW}) gives a strict representation $(\lambda,\GG^{(0)}*\HH,\hat{\pi})$ of $\BB$
such that $L$ is the integrated form of  $\pi$ on $L^2(\GG^{(0)}*\HH,\lambda)$. By
\cite[Lemma~5.22]{MW}, there is a unitary isomorphism from $H$ onto
$L^2(\GG^{(0)}*\HH,\lambda)$. We identify $H$ with $L^2(\GG^{(0)}*\HH,\lambda)$ and view
$\xi$ as a section of the bundle $\GG^{(0)}*\HH$. Let $\mu$ be the measure on $\GU$ given
by $d\mu(x):=\|\xi(x)\|^2d\lambda(x)$. For each $x\in \GU$, define
$\psi_x:\CGx\rightarrow \C$ by
\begin{equation}\label{smaller-trace}
\psi_x(a\cdot\delta_u)=\|\xi(x)\|^{-2}\big(\pi(a)\xi(x),\xi(x)\big)
\end{equation}
where $u\in \Gx$ and $a\in B(u)$. We first show that $\psi_x$ is a state on $\CGx$:

Fix $u\in \Gx$ and $a\in B(u)$. A computation using the multiplication and the involution
formulas  \eqref{multi-formula} and \eqref{invo-formula} shows that for $v\in \Gx$ and
$b\in B(u)$ we have
\begin{equation}\label{multi-algebra-isotropy}
(a\cdot\delta_u)*(b\cdot\delta_v)=ab\cdot \delta_{uv} \text{ and  } (a\cdot \delta_u)^*=a^*\cdot \delta_{u^{-1}}.
\end{equation}
Therefore  using (S1) and (S2) at the final line we see that
\begin{align*}
\psi_x\big((a\cdot\delta_u)*(a\cdot\delta_u)^*\big)&=\psi_x\big(aa^*\cdot\delta_{uu^{-1}}\big)\\
&=\|\xi(x)\|^{-2}\big(\pi(aa^*)\xi(x)\big|\xi(x)\big)\\
&=\|\xi(x)\|^{-2}\big(\pi(a^*)\xi(x)\big|\pi(a^*)\xi(x)\big)\\
&\geq 0
\end{align*}
Since $\hat{\pi}$ is a $*$-functor, (S1)--(S3) imply that $\pi(\one_x)=1_{B(\HH(x))}$.
Now the computation
\[\psi_x(\one_x\cdot\delta_x)=\|\xi(x)\|^{-2}\big(\pi(\one_x)\xi(x)\big|\xi(x)\big)=1\]
implies that $\psi_x$ is a state on $\CGx$.

We  claim that the pair $\big(\mu,\{\psi_x\}_{x\in\GU}\big)$ satisfies the equation
\eqref{Nesh-formula}.  By \eqref{integratedform} for all $f\in\PCG$ we have
\begin{equation}\label{integ-form-psi}
\psi(f)=\big(L(f)\xi\,\big|\,\xi\big)=\int_{\GU}\sum_{u\in \GG^x}\big(\pi(f(u))\xi(s(u)) \,\big| \,\xi(x)\big)\Delta_\lambda(u)^{-\frac{1}{2}}\,d\lambda(x).
\end{equation}
To prove \eqref{Nesh-formula}, it suffices to show that for $\lambda$-almost every $x\in \GU$ we
have
\[\sum_{u\in \GG^x\setminus\GG^x_x}\big(\pi(f(u))\xi(s(u)) \,\big| \,\xi(x)\big)\Delta_\lambda(u)^{-\frac{1}{2}}=0.\]
Equivalently we can show that  for $\lambda$-almost every $x\in \GU$, for each  bisection
$U\subseteq \GG\setminus \cup_{x\in \GU}\GG^x_x$ such that  $u\in \GG^x\cap U$ and that
$a\in B(u)$, we have
\begin{equation}\label{equivalent-formula}
\big(\pi(a)\xi(s(u)) \,\big| \,\xi(x)\big)=0.
\end{equation}

Fix a bisection $U\subseteq \GG\setminus \cup_{x\in \GU}\Gx$ and $g\in \Gamma_c(U;\BB)$
with $\supp g\subseteq U$. Since $r(\supp g)$ is a closed subset of the open set $r(U)$,
there is a positive function $p\in C_0(r(U))\subseteq C_0(\GU)$ such that $p\equiv 1$ on
$r(\supp g)$ and zero otherwise. Fix an approximate identity $(e_\kappa)$ for $\Gzero$
and set $h_\kappa:=p e_\kappa$ for all $\kappa$. Since each $h_\kappa$ is in the
centraliser of $\psi$ and $h_\kappa* g\rightarrow g$, we have
\[\psi(g)=\lim_{\kappa}\psi(g*h_\kappa)=\lim_{\kappa}\psi(h_\kappa*g)=0.\]
Define $q:\GU\rightarrow \C$ by
\[q(x)=\sum_{u\in \GG^x}\overline{\big(\pi(g(u))\xi(s(u)) \,\big| \,\xi(x)\big)\Delta_\lambda(u)^{-\frac{1}{2}}}.\]
Clearly $q(x)\in \C$. Since $\psi(g)=0$,  $\psi(q(x)g)=0$ for all $x\in \GU$. Applying
\eqref{integ-form-psi} for $\psi$ together with (S1) for the $*$-functor $\hat{\pi}$ give
us
\begin{align}\label{formula-thm1-1}
\notag0&=\psi(q(x)g)=\int_{\GU}q(x)\sum_{u\in \GG^x}\big(\pi(g(u))\xi(s(u)) \,\big| \,\xi(x)\big)\Delta_\lambda(u)^{-\frac{1}{2}}\,d\lambda(x)\\
&=\int_{\GU}\Big|\sum_{u\in \GG^x}\big(\pi(g(u))\xi(s(u)) \,\big| \,\xi(x)\big)\Delta_\lambda(u)^{-\frac{1}{2}}\Big|^2\,d\lambda(x).
\end{align}
Thus $\sum_{u\in \GG^x}\big(\pi(g(u))\xi(s(u)) \,\big|
\,\xi(x)\big)\Delta_\lambda(u)^{-\frac{1}{2}}=0$ for $\lambda$-almost every $x\in \GU$.

Since $\BB$ has enough sections, we can fix a  countable set $\{g_n\}$ of elements
$\Gamma_c(U;\BB)$ such that for each $u\in U$, the set $\{g_n(u):n\in \N\}$ is a dense
subset of $B(u)$. Notice that for any $x\in r(U)$ the set $U\cap \GG^x$ is a singleton.
Let $U\cap \GG^x:=u^x$. For each $n\in \N$, let
\[X_n:=\Big\{x\in U:\sum_{u\in \GG^x\cap U}\big(\pi(g_n(u))\xi(s(u)) \,\big| \,\xi(x)\big)\neq 0\Big\} \,\text {  and  }\, X=\bigcup_{n\in \N} X_n.\]
Equation \eqref{formula-thm1-1} implies that $\lambda(X)=0$. For $x\in U\setminus X$ and
$n\in \N$,
\[\big(\pi(g_n(u^x))\xi(s(u^x)) \,\big| \,\xi(x)\big)=\sum_{u\in \GG^x\cap U}\big(\pi(g_n(u))\xi(s(u)) \,\big| \,\xi(x)\big)=0\]
by definition of $x$.   By choice of $g_n$, the set $\{g_n(u^x):n\in \N\}$ is a dense
subset of $B(u^x)$. It follows that $\big(\pi(a)\xi(s(u^x)) \,\big| \,\xi(x)\big)=0$ for
all $a\in B(u^x)$, giving \eqref{equivalent-formula}. So $\psi$ is given by
\eqref{Nesh-formula}.

To see that each $\psi_x$ is a trace, note that since $\Gzero$ is contained in the
centraliser of $\psi_x$, the formula \eqref{Nesh-formula} implies that
\begin{align*}
\int_{\GU}\sum_{u\in \GG^x_x}&\big(\pi\big((f*g)(u)\big)\xi(x) \,\big|\, \xi(x)\big)\Delta_\lambda(u)^{-\frac{1}{2}}\,d\lambda(x)\\
&=\int_{\GU}\sum_{u\in \GG^x_x}\big(\pi\big((g*f)(u)\big)\xi(x))\, \big|\, \xi(x)\big)\Delta_\lambda(u)^{-\frac{1}{2}}\,d\lambda(x)
\end{align*}
for all $f,g\in \Gzero$. Therefore   for $\lambda$-almost every $x\in \GU$, we have
\begin{equation}\label{vir12}
\sum_{u\in \GG^x_x}\big(\pi\big((f*g)(u)\big)\xi(x))\, \big|\, \xi(x)\big)
=\sum_{u\in \GG^x_x}\big(\pi\big((g*f)(u)\big)\xi(x))\, \big|\, \xi(x)\big).
\end{equation}
Fix $a,b\in B(x)$ and $u,v\in \Gx$ so that $a\cdot \delta_u,\, b\cdot \delta_v$ are
typical spanning elements of $\CGx$. Choose $f,g\in \Gzero$ such that  $f(x)=a$ and
$g(x)=b$. Then sums in both sides of \eqref{vir12} collapse and we get
\[\big(\pi(ab)\xi(x) \big| \xi(x)\big)
=\big(\pi(ba)\xi(x) \big| \xi(x)\big) \text { for } \lambda\text{-a.e. } x\in \GU.\]
Since $(a\cdot \delta_u) * (b\cdot \delta_v)=ab\cdot \delta_{uv},$ the formula
\eqref{smaller-trace} for $\psi_x$ implies that
\[\psi_x\big((a\cdot \delta_u) * (b\cdot \delta_v)\big)=\psi_x\big((b\cdot \delta_v) * (a\cdot \delta_u)\big).\]
Thus $\psi_x$ is a trace on $\CGx$. We have now proved that  every state of
$C^*(\GG,\BB)$ is given by a quasi-invariant measure $\mu$ and some $\mu$-measurable field $\{\psi_x\}_{x\in
\GU}$. By
 Lemma~\ref{injectivity},  we see that  each state of $C^*(\GG,\BB)$ is given by \eqref{Nesh-formula} for some $\big(\mu,[\Psi]_\mu\big)$.

We must show that every $\big(\mu,[\Psi]_\mu\big)$ gives a state. It suffices to prove
this for a  probability measure $\mu$ on $\GU$  and a representative $\{\psi_x\}_{x\in \GU}\in
[\Psi]_\mu$. For each $x\in \GU$, define $\varphi_x:\PCG\rightarrow \C$ by
\[\varphi_x(f)=\sum_{u\in \Gx}\psi_x\big(f(u)\cdot\delta_u\big) \text{ for } f\in \PCG.\]
Since $\psi_x$ is a state $\varphi_x$ extends to $\CG$. Now if we prove that each
$\varphi_x$ is a well-defined state on $\CG$, then since $x\mapsto
\varphi_x(b):\GU\rightarrow \C$ is $\mu$-measurable for all $b\in \CG$, we can define a
functional $\psi$ on $\CG$
 by $\psi(b)=\int_{\GU}\varphi_x(b)\,d\mu(x)$. Since $\mu$ is a probability measure on $\GU$, $\psi$ is a state on $\CG$ which satisfies \eqref{Nesh-formula}.  So we fix  $x\in \GU$ and show that $\varphi_x$ is a well-defined state on $\CG$:

Let $(H_x,\pi_x,\zeta_x)$ be the GNS-triple corresponding to $\psi_x$. Let $Y(x)$ be the
closure of $\Gamma_c(\GG_x;\BB)$ under the $\CGx$-valued pre-inner product
\[\langle f,g\rangle=f^**g.\]
Then $Y(x)$ is a right Hilbert $\CGx$-module with the right action determined by the
multiplication, see \cite[Lemma~2.16]{tfb}. Also $\CG$ acts as adjointable operator on
$Y(x)$ by multiplication. By \cite[Proposition~2.66]{tfb} there is a representation
$Y(x)$-$\Ind(\pi_x):\CG\rightarrow \LL\big(Y(x)\otimes_{\CGx}H_x\big)$ such that
\[Y(x)\text{-}\Ind(\pi_x)(f)(g\otimes k)=f*g\otimes k.\]

For convenience, we write $\theta_x:=Y(x)$-$\Ind(\pi_x)$. Take $h_x\in \PCGx$ such that
$\supp h_x \subseteq \{x\}$ and $h_x(x)=\one_x$. We take $f\in \PCG$ and compute
$\theta_x(f)(h_x\otimes \zeta_x)$:

\begin{align}\label{computation-of-theta}
\big(\theta_x(f)(h_x\otimes \zeta_x) \,\big|\,(h_x\otimes \zeta_x)\big)\notag&=\big(f*h_x\otimes \zeta_x \,\big|\,h_x\otimes \zeta_x\big)\\
\notag&=\big(\pi_x(\langle h_x,f*h_x\rangle) \zeta_x \,\big|\,\zeta_x\big)\\
&=\psi_x\big(\langle h_x,f*h_x\rangle\big).
\end{align}
For each $u\in \Gx$, we have
\[\langle h_x,f*h_x\rangle(u)=(h_x^**f*h_x)(u)=\sum_{\alpha\beta\gamma=u}h_x(\alpha^{-1})^*f(\beta)h_x(\gamma).\]
Each summand  vanishes unless $\alpha^{-1}=\gamma=x$ and $\beta=u$. Therefore
\[\langle h_x,f*h_x\rangle(u)=\one_x^*f(u)\one_x=f(u),\]
and hence $\langle h_x,f*h_x\rangle=f|_{\Gx}$. Putting this in
\eqref{computation-of-theta}, we get
\[\big(\theta_x(f)(h_x\otimes \zeta_x) \,\big|\,(h_x\otimes \zeta_x)\big)=\psi_x\big(f|_{\Gx}\big)=\sum_{u\in \Gx}\psi_x\big(f(u)\cdot \delta_u\big)=\varphi(f).\]
Also since $\langle h_x,h_x\rangle=\one_x\cdot\delta_x$,
\[\|h_x\otimes\zeta_x\|=\big(h_x\otimes\zeta_x\big|h_x\otimes\zeta_x\big)=\big( \pi_x\langle h_x,h_x\rangle\zeta_x\big|\zeta_x\big)=\psi_x\big(\langle h_x,h_x\rangle\big)=\psi_x(\one_x\cdot\delta_x)=1.\]
Now  since $f\mapsto \big(\theta_x(f)(h_x\otimes \zeta_x) \,\big|\,(h_x\otimes
\zeta_x)\big)$ is a state,  $\varphi$ is a state as well. Thus there is surjection
between the simplex of the states on $\CG$ with centraliser containing $\Gzero$ and the
pairs $\big(\mu,[\psi]_{\mu}\big)$.  Lemma~\ref{injectivity} gives the
injectivity and we have now completed the proof.
\end{proof}

\begin{definition}
Given a pair $(\mu, C)$ consisting of a probability measure $\mu$ on $\GU$ and a
$\mu$-equivalence class $C$ of $\mu$-measurable fields of tracial states, we write
$\tilde{\Theta}(\mu, C)$ for the state $\Theta(\mu, \Psi)$ for any representative $\Psi$
of $C$.
\end{definition}

\begin{thm}\label{thm2}
Let $p:\BB\rightarrow \GG$ be a Fell bundle with singly generated fibres over a locally
compact second countable \'{e}tale groupoid $\GG$.  Suppose that $\gamma\mapsto
\one_\gamma:\GG\rightarrow \BB$ is continuous. Let $D$ be a continuous $\R$-valued
$1$-cocycle on $\GG$ and let $\tau$ be the dynamic on $\CG$ given by
$\tau_t(f)(\gamma)=e^{itD(\gamma)}f(\gamma)$. Let $\beta\in \R$. Then
$\widetilde{\Theta}$ restricts to  a bijection between the simplex of  KMS$_\beta$ states
of $(\CG,\tau)$  and the pairs $\big(\mu,[\Psi]_{\mu}\big)$ such that
\begin{itemize}
\item [(I)] $\mu$ is a quasi-invariant measure with Radon--Nykodym cocycle $e^{-\beta
    D}$, and
\item [(II)] for $\mu$-almost every $x\in \GU$, we have
\[
\psi_{s(\eta)}(a\cdot \delta_u)=\psi_{r(\eta)}\big((\one_\eta a \one_\eta^*)\cdot \delta_{\eta u\eta^{-1}}\big) \quad\text{for   }u\in \Gx, a\in B(u)\text{  and  }\eta\in \GG_x.
\]
\end{itemize}
\end{thm}

\begin{remark}
In principal, the condition in Theorem~\ref{thm2}(II) depends on the particular
representative $\Psi = \{\psi_x\}_{x\in \GU}$ of the $\mu$-equivalence class $[\Psi]_\mu$. But if
$\Psi = \{\psi_x\}_{x\in \GU}$ and $\Psi' = \{\psi'_x\}$ represent the same equivalence class, then
$\psi_x = \psi'_x$ for $\mu$-almost every $x$, and so $\Psi$ satisfies~(II) if and only if
$\Psi'$ does.
\end{remark}

Before starting the proof, we establish some notation. Let $U$ be a bisection. For each
$x\in \GU$,  we write $U^x:=r^{-1}(x)\cap U$ and $U_x:=s^{-1}(x)\cap U$. The maps
$x\mapsto U^x:r(U)\rightarrow U$ and $x\mapsto U_x:s(U)\rightarrow U$ are homeomorphisms
and we can view them as the inverse of $r,s$ respectively. We also write
$T_U:r(U)\rightarrow s(U)$ for the homeomorphism given by $T_U(x)=s(U^x)$.
\begin{proof}
Suppose that $\psi$ is a KMS$_\beta$ state on $(\CG,\tau)$. Since $D|_{\GU}=0$, the KMS
condition implies that $\Gzero$ is contained in the centraliser of $\psi$. By
Theorem~\ref{thm1} there is a pair $\big(\mu,[\Psi]_\mu\big)$ consisting of a
probability measure $\mu$ on $\GU$ and a $\mu$-equivalence class $[\Psi]_\mu$ of
$\mu$-measurable fields of tracial states on $\CGx$ that satisfies \eqref{Nesh-formula}.
Fix a representative $\{\psi_x\}_{x\in \GU}\in[\Psi]_\mu$. We claim that
$\mu$ and $\{\psi_x\}_{x\in \GU}$ satisfy (I)-(II).

First note that for a bisection $U$, $f\in \Gamma_c(U;\BB)$ and $g\in \PCG$, the
multiplication formula in $\PCG$ implies that
\begin{align*}
f*g(\gamma)=\sum_{\alpha\beta=\gamma}f(\alpha)g(\beta)=
\begin{cases}
f(U^x)g((U^x)^{-1}\gamma)&\text{if } x=r(\gamma)\in r(U)\\
	0&\text{if } r(\gamma)\notin r(U).\\
	\end{cases}
	\end{align*}
Similarly
 \begin{align*}
g*\tau_{i\beta}(f)(\gamma)&=\begin{cases}
e^{-\beta D(U_x)}g(\gamma(U_x)^{-1})f(U_x)&\text{if }x=s(\gamma)\in s(U)\\
	0&\text{if }s(\gamma)\notin s(U).\\
	\end{cases}
	\end{align*}
	Since $\psi$ is a KMS$_\beta$ state, we have $\psi(f*g)=\psi(g*\tau_{i\beta}(f))$.
Now applying formula \eqref{Nesh-formula} for $\psi$ gives us
	\begin{align}\label{formula-thm2-1}
\int_{r(U)} \sum_{u\in \Gx}\psi_x\notag&\big(f(U^x)g((U^x)^{-1}u)\cdot\delta_u\big)\, d\mu(x)\\
&=\int_{s(U)} e^{-\beta D(U_x)}\sum_{u\in \Gx}\psi_x\big(g(u(U_x)^{-1})f(U_x)\cdot\delta_u\big)\, d\mu(x).
	\end{align}
		
		To see (I), fix a bisection $U$ and let $q\in C_c(s(U))$. Since $\BB$ has enough
sections, we can define $h:U\rightarrow \BB$ by $h(\gamma)=q(s(\gamma))\one_\gamma$.
Since $\gamma\mapsto \one_\gamma$ is continuous,  $h$ descends to a continuous section
$\tilde{h}$ on $\GG$. Now we apply \eqref{formula-thm2-1} with $f:=\tilde{h}$ and
$g:=\tilde{h}^*$. The sums in both sides collapse to the single term $u=x$. Since
$U^x=U_{T_U(x)}$, we have
\[\int_{r(U)} \psi_x\Big(\big(|q(T_U(x))|^2\one_x\one_x^*\big)\cdot\delta_x\Big)\, d\mu(x)\\
=\int_{s(U)} e^{-\beta D(U_x)}\psi_x\big((|q(x)|^2\one_x\one_x^*)\cdot\delta_x\big)\, d\mu(x).\]
 Note that $(\lambda a)\cdot \delta=\lambda(a\cdot \delta)$ for all $\lambda\in \C$ and $\one_x\one_x^*=1_{A(x)}=\one_x$. Since  $\one_x\cdot \delta_x=1_{\CGx}$ and  $\psi_x$ is a state on $\CGx$, we have
\[\int_{r(U)} \big|q(T_U(x))\big|^2\, d\mu(x)\\
=\int_{s(U)} e^{-\beta D(U_x)}|q(x)|^2\, d\mu(x).\]
Thus $\mu$ is a quasi-invariant measure with Radon--Nykodym cocycle $e^{-\beta D}$.

For (II), let $x\in \GU$,  $u\in \Gx, a\in B(u)$ and   $\eta\in \GG_x$. Let $\tilde{a}\in
\Gamma_c(\GG_x;\BB)$ such that $ \tilde{a}$ is supported in a bisection $U$ and
$\tilde{a}(u)=a$. Since $U$ is a bisection, it follows that $\tilde{a}(v)=0$ for all
$v\in \GG_x\setminus \{u\}$. Fix a bisection $V$ containing $\eta$ such that
$s(V)\subseteq s(U)$. Fix $q\in C_c(\GU)$ such that $q\equiv1$ on a neighborhood of $x$
and $\supp (q)\subseteq s(V)$. Define $h\in \PCG$ by
\begin{align*}
h(\gamma)=\begin{cases}
q(s(\gamma))\one_\gamma&\text{if $\gamma\in V$}\\
	0&\text{otherwise.}\\
	\end{cases}
\end{align*}
Since $\psi$ is a KMS$_\beta$ state, we have
\begin{equation} \label{kmsforii}
\psi(( \tilde{a}*h^*)*h)=\psi(h*\tau_{i\beta}( \tilde{a}*h^*)).
\end{equation}

We compute both sides of \eqref{kmsforii}.  For the left-hand side, we first apply the
formula \eqref{Nesh-formula} for $\psi$ to get
\begin{equation} \label{simplifying-left-kmsforii}
\psi(( \tilde{a}*h^*)*h) = \int_{\GU} \sum_{v\in \GG^y_y}\psi_y\big(( \tilde{a}*h^**h)(v)\cdot\delta_v\big)\,d\mu(y).
\end{equation}
Since $h$ is supported on the bisection $V$, $h^**h$ is supported on $s(V)$ and we have
\[( \tilde{a}*h^**h)(v)=\sum_{\alpha\beta=v} \tilde{a}(\alpha)(h^**h)(\beta)= \tilde{a}(v)(h^**h)(s(v)).\]
 Since $\tilde{a}$ is supported in $U$,
\begin{align*}
\sum_{v\in \GG^y_y}\psi_y\big(( \tilde{a}*h^**h)(v)\cdot\delta_v\big)&=\sum_{v\in \GG^y_y\cap U}\psi_y\big( \big(\tilde{a}(v)(h^**h)(s(v))\big)\cdot\delta_v\big)\\
 &=\psi_y\big( \big(\tilde{a}(U_y)(h^**h)(y)\big)\cdot\delta_{U_y}\big).
\end{align*}
Putting this in \eqref{simplifying-left-kmsforii} and applying the definition of  $h$, we
get
\begin{align} \label{kmsforiil}
\psi(( \tilde{a}*h^*)*h)\notag&=\int_{s(V)} \psi_y\big(\big( \tilde{a}(U_y)(h^**h)(y)\big)\cdot\delta_{U_y}\big)\,d\mu(y)\\
&=\int_{s(V)}|q(y)|^2\psi_y\big( \tilde{a}(U_y)\cdot\delta_{U_y}\big)\,d\mu(y).
\end{align}

For the right-hand side, we start by applying the formula \eqref{Nesh-formula} for $\psi$:
\begin{align*}
\psi\big(h*\tau_{i\beta}( \tilde{a}*h^*)\big)\notag&=\int_{\GU} \sum_{w\in \GG^z_z}\psi_z\big((h*\tau_{i\beta}( \tilde{a}*h^*))(w)\cdot\delta_w\big)\,d\mu(z).
\end{align*}
Two applications of the multiplication formula in $\PCG$ give
\begin{align*}
\psi&\big(h*\tau_{i\beta}( \tilde{a}*h^*)\big)=\int_{r(V)} \sum_{w\in \GG^z_z} \psi_z\Big(\big(h(V^z)\tau_{i\beta}( \tilde{a}*h^*)((V^z)^{-1}w)\big)\cdot\delta_w\Big)\,d\mu(z)\\
\!&=\!\int_{r(V)} \!\!\! e^{-\beta D\big(U_{T_V(z)}(V^z)^{-1}\big)}\psi_z\Big(\big(h(V^z) \tilde{a}\big(U_{T_V(z)}\big)h(V^z)^*\big)\cdot\delta_{V^zU_{T_V(z)}(V^z)^{-1}}\Big)\,d\mu(z)\\
\!&=\!\int_{r(V)} \!\!\!e^{-\beta D\big( U_{T_V(z)}(V^z)^{-1}\big)}\big|q(T_V(z))\big|^2\psi_z\Big(\big(\one_{V^z} \tilde{a}
(U_{T_V(z)})\one_{V^z}^*\big)\cdot\delta_{V^zU_{T_V(z)}(V^z)^{-1}}\Big)\,d\mu(z).
\end{align*}
Since for each $z\in r(V)$, we have $V^z=V_{T_V(z)}$ and $z=r\big(V_{T_V(z)}\big)$, the
variable substitution  $y=T_V(z)$ gives
\begin{equation} \label{kmsforiir}
\psi(h*\tau_{i\beta}( \tilde{a}*h^*)) = \int_{s(V)} |q(y)|^2\psi_{r(V_y)}\big((\one_{V_y} \tilde{a}(U_{y})\one_{V_y}^*)\cdot\delta_{V_yU_{y}(V_y)^{-1}}\big)\,d\mu(y).
\end{equation}
Putting $y=x$ in \eqref{kmsforiir}, we have $U_y=u$ and $V_y=\eta$. Since $|q(x)|^2=1$,
part (II) follows from \eqref{kmsforiil} and \eqref{kmsforiir}.

For the other direction, suppose that  $\big(\mu,[\psi]_{\mu}\big)$ satisfies (I)-(II).
The formula \eqref{Nesh-formula} in Theorem~\ref{thm1} gives a state $\psi:=\Theta(\mu,\Psi)$ on $\CG$. We
aim to show that $\psi$ is a KMS$_\beta$ state. It suffices to show that for each
bisection $U$, each $f\in \Gamma_c(U;\BB)$, and each   $g\in \PCG$ we have
\begin{equation} \label{kms-check}
\psi(f*g)=\psi(g*\tau_{i\beta}(f))
\end{equation}
 Fix a representative $\{\psi_x\}_{x\in \GU}\in[\Psi]_\mu$. The  left-hand side of \eqref{kms-check} is
\begin{equation}\label{lhand-kms-check}
\psi(f*g)=\int_{r(U)}\sum_{u\in \Gx}\psi_x\big(\big(f(U^x)g((U^x)^{-1}u)\big)\cdot\delta_u\big)\,d\mu(x).
\end{equation}
We compute the right-hand in terms of the representative $\{\psi_x\}_{x\in
\GU}\in[\Psi]_\mu$. We start by the multiplication formula in $\PCG$ and the formula
\eqref{Nesh-formula} for $\psi$:
\begin{align*}
\psi(g*\tau_{i\beta}(f))&=\int_{x\in \GU}\sum_{u\in \Gx}\psi_x\big((g*\tau_{i\beta}(f))(u)\cdot\delta_u\big)\,d\mu(x)\\
&=\int_{x\in s(U)}\sum_{u\in \Gx}e^{-\beta D(U_x)}\psi_x\big(\big(g(u(U_x)^{-1})f(U_x)\big)\cdot\delta_u\big)\,d\mu(x).
\end{align*}
Since $\mu$ is quasi-invariant with Radon--Nykodym cocycle $e^{-\beta D}$,  the
substitution $x=T_U(y)$  gives
\[\psi(g*\tau_{i\beta}(f))=\int_{r(U)}\sum_{u\in \GG^{T_U(y)}_{T_U(y)}}\psi_{{T_U(y)}}\big(\big(g(u(U_{{T_U(y)}})^{-1})f(U_{{T_U(y)}})\big)\cdot\delta_u\big)\,d\mu(y).\]
Since $U_{{T_U(y)}}=U^y$, we obtain
\[\psi(g*\tau_{i\beta}(f))=\int_{r(U)}\sum_{u\in \GG^{T_U(y)}_{T_U(y)}}\psi_{{T_U(y)}}\big(\big(g(u(U^y)^{-1})f(U^y)\big)\cdot\delta_u\big)\,d\mu(y).\]
Applying the identity $\GG^{T_U(y)}_{T_U(y)}(U^y)^{-1}=(U^y)^{-1}\GG^y_y$, we can rewrite
the sum as
\begin{equation} \label{rhand-kms-check}
\psi(g*\tau_{i\beta}(f))=\int_{r(U)}\sum_{v\in \GG^{y}_{y}}\psi_{{T_U(y)}}\Big(\big(g((U^y)^{-1}v)f(U^y)\big)\cdot\delta_{(U^y)^{-1}vU^y}\Big)\,d\mu(y).
\end{equation}

  To simplify further,  fix $v\in \GG^{y}_{y}$. Using that $\one_{U^y}\one_{U^y}^*=\one_y$ then applying \eqref{multi-algebra-isotropy} give
\begin{align*}
\psi_{{T_U(y)}}&\Big(\big(g((U^y)^{-1}v)f(U^y)\big)\cdot\delta_{(U^y)^{-1}vU^y}\Big)\\
&=\psi_{{T_U(y)}}\Big(\big(g((U^y)^{-1}v)\one_{U^y}\one_{U^y}^*f(U^y)\big)\cdot\delta_{(U^y)^{-1}vU^y(U^y)^{-1}U^y}\Big)\\
&=\psi_{{T_U(y)}}\bigg(\Big(\big(g((U^y)^{-1}v)\one_{U^y}\big)\cdot\delta_{(U^y)^{-1}vU^y}\Big)\Big(\big(\one_{U^y}^*f(U^y)\big)\cdot\delta_{(U^y)^{-1}U^y}\Big)\bigg).
\end{align*}
Since $g((U^y)^{-1}v)\one_{U^y}\in B((U^y)^{-1}vU^y)$ and $(U^y)^{-1}vU^y\in
\GG^{T_U(y)}_{T_U(y)}$, the trace property of $\psi_{T_U(y)}$ implies that
\begin{align*}
\psi_{{T_U(y)}}&\Big(\big(g((U^y)^{-1}v)f(U^y)\big)\cdot\delta_{(U^y)^{-1}vU^y}\Big)\\
&=\psi_{{T_U(y)}}\bigg(\Big(\big(\one_{U^y}^*f(U^y)\big)\cdot\delta_{(U^y)^{-1}U^y}\Big)\Big(\big(g((U^y)^{-1}v)\one_{U^y}\big)\cdot\delta_{(U^y)^{-1}vU^y}\Big) \bigg)\\
&=\psi_{{T_U(y)}}\Big(\big(\one_{U^y}^*f(U^y)g((U^y)^{-1}v)\one_{U^y}\big)\cdot\delta_{(U^y)^{-1}vU^y}\Big) \quad \text {by \eqref{multi-algebra-isotropy}}.
\end{align*}
We apply (II) with $\eta=U^y$. Recall that $T_U(y)=s(U^y)$ and so $r(\eta)=y$ and we have
\begin{align*}
\psi_{T_U(y)}\Big(\big(g((U^y)^{-1}v)&f(U^y)\big)\cdot\delta_{(U^y)^{-1}vU^y}\Big)\\
&=\psi_{y}\Big(\big(\one_{U^y}\one_{U^y}^*f(U^y)g((U^y)^{-1}v)\one_{U^y}\one_{U^y}^*\big)\cdot\delta_{v}\Big)\\
&=\psi_{y}\big(\big(f(U^y)g((U^y)^{-1}v)\big)\cdot\delta_v\big).
\end{align*}
Substituting this in each term of \eqref{rhand-kms-check} gives
\begin{align*}
\psi(g*\tau_{i\beta}(f))=\int_{r(U)}\sum_{v\in \GG^{y}_{y}}\psi_{y}\big(\big(f(U^y)g((U^y)^{-1}v)\big)\cdot\delta_v\big)\,d\mu(y).
\end{align*}
which is precisely  \eqref{lhand-kms-check}.  So \eqref{kms-check} holds,  and  $\psi$ is
a KMS$_\beta$ state for $\tau$.
\end{proof}

\begin{lemma}
With the hypotheses of Theorem~\ref{thm2}, suppose that  $\psi$ is a KMS$_\beta$ state on
$(\CG,\tau)$ and that  $\big(\mu, C\big)$ is the associated pair given by
Theorem~\ref{thm2}. For any $\mu$-measurable field of states $\Psi = \{\psi_x\}_{x\in \GU}$ such that
$[\Psi]_\mu = C$, we have
\[
    \psi_{x}(a\cdot \delta_u)=0 \text{ for } \mu\text{-a.e. }x\in \GU \text{, }u\in \Gx\setminus D^{-1}(0), \text{ and }a\in B(u).
\]
\end{lemma}
\begin{proof}
Fix $x\in \GU$, $u\in \Gx\setminus D^{-1}(0)$ and $a\in B(u)$. Let
$\varepsilon:=\frac{|D(u)+1|}{2}$. Since $\BB$ has enough sections, there exists  $f\in
\PCG$ such that $f(u)=a$ and that $f$ is supported in a bisection $U$ such that
$D(U)\subseteq \big(D(u)-\epsilon, D(u)+\epsilon\big)$. In particular, if $D(u)<0$, then
$D(v)<-\varepsilon$ for all $v\in U$, and if $D(u)>0$, then  $D(v)>\varepsilon$ for all
$v\in U$. Recall that $U_x$ is the unique element of $s^{-1}(x)\cap U$. Since $\psi$ is a
KMS$_\beta$ state, it is $\tau$-invariant  and we have $\psi(\tau_1(f))=\psi(f)$.
Applying the  formula \eqref{Nesh-formula} for $\psi$ gives
\[\int_{s(U)}\psi_x\big(f(U_x)\cdot\delta_{U_x}\big)\,d\mu(x)=\int_{s(U)}e^{-D(U_x)}\psi_x\big(f(U_x)\cdot\delta_{U_x}\big)\,d\mu(x).\]
Now our choice of $u$  forces $\psi_x\big(f(U_x)\cdot\delta_{U_x}\big)=0$ for $\mu$-almost every
$x\in \GU$. In particular $\psi_x(a\cdot\delta_{u})=0$  for $\mu$-almost every $x\in \GU$.
\end{proof}

By specialising to $\beta = 0$, we can use our results to describe the trace space of the
cross-section algebra of a Fell bundle with singly generated fibres. This is particularly
important given the role of the trace simplex of a simple $C^*$-algebra in Elliott's
classification program.

\begin{cor}
Let $p:\BB\rightarrow \GG$ be a Fell bundle with singly generated fibres over a locally
compact second countable \'{e}tale groupoid $\GG$.  Then $\widetilde{\Theta}$ restricts
to a bijection between the trace space of $(\CG,\tau)$  and the pairs
$\big(\mu,[\Psi]_{\mu}\big)$ consisting of a probability measure $\mu$ on $\GU$ and a
$\mu$-equivalence class $[\Psi]_\mu$ of $\mu$-measurable fields of tracial states on
$\CGx$ such that
 \begin{itemize}
\item [(I)] $\mu$ is a quasi-invariant measure with Radon--Nykodym cocycle $1$.
\item [(II)] For $\mu$-almost every $x\in \GU$, we have
\[\psi_{s(\eta)}(a\cdot \delta_u)=\psi_{r(\eta)}\big((\one_\eta a \one_\eta^*)\cdot \delta_{\eta u\eta^{-1}}\big) \quad\text{for   }u\in \Gx, a\in B(u)\text{  and  }\eta\in \GG_x.\]
\end{itemize}
\end{cor}
\begin{proof}
The  KMS condition at inverse temperature $0$ reduces to the trace property. So we just
need to observe that the proof of Theorem~\ref{thm2} does not require the automatic
$\tau$-invariance of KMS states for $\tau$.
\end{proof}

\section{KMS states  on  twisted groupoid $C^*$-algebras}\label{sec:KMS twisted
G}

To apply our results to twisted groupoid $C^*$-algebras, we recall how to regard a
twisted groupoid $C^*$-algebra as the cross-sectional algebra of a Fell-bundle with
one-dimensional fibres. This is standard; we just include it for completeness.

\begin{lemma}\label{Fell-groupoid}
Let $\GG$ be a locally compact second countable \'{e}tale groupoid, and let $\sigma\in
Z^2(\GG,\T)$. Let  $\BB:=\GG\times \C$ and equip $\BB$ with the product topology. Define
$p:\BB\rightarrow \GG$ by $p(\gamma,z)=\gamma$. Then
\begin{itemize}
\item[(I)] $p:\BB\rightarrow \GG$ is a Fell bundle with respect to the multiplication
    and involution given by
\begin{equation}\label{multiplicationFG}
(\gamma,z)(\eta,w)=(\gamma\eta,\sigma(\gamma,\eta)zw)\text{ and } (\gamma,z)^*=(\gamma^{-1},\overline{\sigma(\gamma,\gamma^{-1})}\overline{z}).
 \end{equation}
\item[(II)]For each $\gamma\in \GG$, the fibre $\BB(\gamma)$ is singly generated with
    $\one_\gamma:=(\gamma,1)$. The map $\gamma\mapsto \one_\gamma:\GG\rightarrow \BB$
    is continuous.
\item[(III)] There is an injective $*$-homomorphism $\Phi$ from $C_c(\GG,\sigma)$
    onto $\Gamma_c(\GG,B)$   such that
\[
\Phi(f)(\gamma)=(\gamma,f(\gamma)) \text { for all } f\in C_c(\GG,\sigma) \text{ and }\gamma\in \GG.
\]
This homomorphism extends to an isomorphism $\Phi:C^*(\GG,\sigma)\rightarrow
C^*(\GG,B)$.
\item[(IV)] There is  an isomorphism $\Upsilon:C^*(\GG_x^x,\sigma)\rightarrow \CGx$ such
    that
\[\Upsilon(W_u)= (u,1)\cdot \delta_{u} \text{ for all } u\in \GG^x_x. \]
\end{itemize}
\end{lemma}
\begin{proof}
For (I), since $\C$ is a Banach space, $\BB$ is the trivial upper-semi continuous Banach
bundle. We check (F1)--(F5):  The conditions (F1) and (F2)  follow from
\eqref{multiplicationFG} easily. To see (F3), let $a:=(\gamma,z)$ and $b:=(\eta,w)$. An
easy computation using \eqref{multiplicationFG} shows that
\[(ab)^*=\big((\eta\gamma)^{-1},\overline{\sigma(\gamma\eta,\eta^{-1}\gamma^{-1})\sigma(\gamma,\eta)}\overline{zw}\big), \text{ and}\]
\[b^*a^*=\big((\eta\gamma)^{-1},\sigma(\eta^{-1},\gamma^{-1}) \overline{\sigma(\eta,\eta^{-1})\sigma(\gamma,\gamma^{-1})}\overline{zw}\big). \]
Two applications of the cocycle relation give us
\begin{align*}
\sigma(\eta^{-1},\gamma^{-1}) \sigma(\gamma\eta,\eta^{-1}\gamma^{-1})\sigma(\gamma,\eta)&=\sigma(\gamma\eta,\eta^{-1})\sigma(\gamma,\gamma^{-1})\sigma(\gamma,\eta)\\&=\sigma(\eta,\eta^{-1})\sigma(\eta,r(\eta))\sigma(\gamma,\gamma^{-1})\\
&=\sigma(\eta,\eta^{-1})\sigma(\gamma,\gamma^{-1}).
\end{align*}
Therefore  $(ab)^*=b^*a^*$. For $(F4)$, let $x\in\GU$. Since $x^{-1}=x=x^{-1}x$, the
operations \eqref{multiplicationFG} make sense in the fibre $B(x)$ and turn it into a
$*$-algebra.
 Also for $a=(x,z)\in B(x)$, we have $\|aa^*\|=|c(x^{-1},x)z\overline{z}|=|z|^2=\|a\|^2$. Thus $B(x)$ is a $C^*$-algebra. For (F5), note that each fibre $B(\gamma)$ is a full left Hilbert  $A(r(\gamma))$-module and a full right Hilbert  $A(s(\gamma))$-module. Equations \eqref{imprimitivity} and \eqref{f5} follow from \eqref{multiplicationFG}.

(II) is clear. To see (III),  note that the multiplication and involution formulas in
$C_c(\GG,\sigma)$ and $\PCG$ show that $\Phi$ is a $*$-homomorphism. Since  each section
$g\in \PCG$ has the form $g(\gamma)=(\gamma, z_{g,\gamma})$ for some $z_{g,\gamma}\in
\C$, we can define $\tilde{\Phi}:\PCG\rightarrow C_c(\GG,\sigma)$ by
$\tilde{\Phi}(g)(\gamma)=z_{g,\gamma}$. An easy computation shows that $\tilde{\Phi}$ is
the inverse of $\Phi$ and therefore $\Phi$ is a bijection. For each $I$-norm decreasing
representation $L$ of $\PCG$, the map $L\circ \Phi$ is a $*$-representation of
$C_c(\GG,\sigma)$. Therefore
\begin{align*}
\|\Phi&(f)\|_{\PCG} \\
    &=\sup \{\|L(\Phi(f))\|:L \text{ is an $I$-norm decreasing representation  of } \PCG\}\\
    &\leq \sup \{\|L'(f)\|:L' \text{ is a $*$-representation  of } C_c(\GG,\sigma)\}\\
    &=\|f\|_{ C_c(\GG,\sigma)}.
\end{align*}
Thus $\Phi$  is norm decreasing and therefore extends to an isomorphism of
$C^*$-algebras.

For (IV), take  $W_u,W_v\in \Gx$. We have
\[\Upsilon(W_uW_v)=\sigma(u,v)\Upsilon(W_{uv})=\sigma(u,v)((uv,1)\cdot \delta_{uv}).\]
To compare this with $\Upsilon(W_u)\Upsilon(W_v)$, we calculate applying
\eqref{multi-algebra-isotropy} in the second equality:
\[\Upsilon(W_u)\Upsilon(W_v)=\big((u,1)\cdot \delta_{u}\big)*\big((v,1)\cdot \delta_{v}\big)=(u,1)(v,1)\cdot \delta_{u,v}=\sigma(u,v)((uv,1)\cdot \delta_{uv}).\]
 Thus $\Upsilon$ is a $*$-homomorphism. The map $\tilde{\Upsilon}:\CGx\rightarrow C^*(\GG_x^x,\sigma)$ given by $\tilde{\Upsilon}((u,z)\cdot\delta_u)=zW_u$ is an inverse for $\Upsilon$, so  $\Upsilon$ descends to an isomorphism of $C^*$-algebras.
\end{proof}

In parallel with Section~\ref{sec:KMS-Fell}, we say that  a collection $\{\psi_x\}_{x\in
\GU}$ of states $\psi_x$ on  $C^*(\GG^x_x,\sigma)$ is a \textit{$\mu$-measurable field of
states} if for every $f\in C_c(\GG,\sigma)$, the function $x\mapsto\sum_{u\in
\Gx}f(u)\psi_x(W_u)$ is $\mu$-measurable.

\begin{cor}\label{kms states for groupoid}
Let $\GG$ be a locally compact second countable \'{e}tale groupoid, and let $\sigma\in
Z^2(\GG,\T)$. Let $D$ be a continuous $\R$-valued $1$-cocycle on $\GG$ and let
$\tilde{\tau}$ be the dynamics on $C^*(\GG,\sigma)$ given by
$\tilde{\tau}_t(f)(\gamma)=e^{itD(\gamma)}f(\gamma)$. Take $\beta\in \R$. There is a
bijection between the simplex of the KMS$_\beta$ states of
$\big(C^*(\GG,\sigma),\tilde{\tau}\big)$  and the pairs $\big(\mu,[\Psi]_{\mu}\big)$
consisting of a probability measure $\mu$ on $\GU$ and a $\mu$-equivalence class
$[\Psi]_\mu$ of $\mu$-measurable fields of tracial states on $C^*(\GG^x_x,\sigma)$ such
that
\begin{itemize}
\item [(I)] $\mu$ is a quasi-invariant measure with Radon--Nykodym cocycle $e^{-\beta
    D}$.
\item [(II)]For each representative $\{\psi_x\}_{x\in \GU}\in[\Psi]_\mu$ and for
    $\mu$-almost every $x\in \GU$, we have
\[\psi_{x}(W_u)=\sigma(\eta u,\eta^{-1})\sigma(\eta,u)\overline{\sigma(\eta^{-1},\eta)}\psi_{r(\eta)}\big(W_{\eta u\eta^{-1}}\big)\quad  \text{for  } u\in \GG^x_x, \text{ and }  \eta\in \GG_x.\]
\end{itemize}
The state corresponding to the pair $\big(\mu,[\Psi]_{\mu}\big)$ is given by
\begin{equation}\label{Nesh-formula-groupoid}
\psi(f)=\int_{\GG^{(0)}}\sum_{u\in \Gx}f(u)\psi_x(W_u)\, d\mu(x)\quad\text{for all }f\in C_c(\GG,\sigma).
\end{equation}
\end{cor}

\begin{proof}
Lemma~\ref{Fell-groupoid} yields a Fell bundle $\BB$ over $\GG$, an  isomorphism
$\Phi:C^*(\GG,\sigma)\rightarrow C^*(\GG,\BB)$, and isomorphism
$\Upsilon:C^*(\GG_x^x,\sigma)\rightarrow \CGx$. The isomorphism $\Phi$ intertwines the
dynamics $\tilde{\tau}$ and $\tau$ induced by $D$ on $C^*(\GG,\sigma)$ and
$C^*(\GG,\BB)$. We aim to apply Theorem~\ref{thm2}.

Let $\psi$ be a KMS$_\beta$ state of $\big(C^*(\GG,\sigma),\tilde{\tau}\big)$. Then
$\varphi:=\psi\circ \Phi^{-1}$ is a  KMS$_\beta$ state on $\big(C^*(\GG,\BB),\tau\big)$
and Theorem~\ref{thm2}  gives a pair  $\big(\mu,\{\varphi_x\}_{x\in\GU}\big)$ of a
probability measure $\mu$ on $\GU$ and a  $\mu$-measurable fields of tracial states on
$C^*(\GG^x_x,\BB)$ satisfying (I)-(II) of Theorem~\ref{thm2}.  Let
$\psi_x:=\varphi_x\circ \Upsilon$. For each $f\in C_c(\GG,\sigma)$, the function
$x\mapsto\sum_{u\in \Gx}f(u)\psi_x(W_u)=\sum_{u\in \Gx}\varphi_x ((u,f(u))\cdot
\delta_{u})$ is $\mu$-measurable. Therefore $\{\psi_x\}_{x\in \GU}$ is a $\mu$-measurable
field of states on $C^*(\GG^x_x,\sigma)$.

To see that $\{\psi_x\}_{x\in \Lambda^\infty}$ satisfies (II), let $u\in \GG_x^x$ and
$\eta\in\GG_x$.
 A computation in $\GG\times \C$ shows that
\[
\one_\eta (u,z) \one_\eta^*=(\eta,1) (u,1) (\eta,1)^*
= \big(\eta u \eta^{-1}, \sigma(\eta u,\eta^{-1})\sigma(\eta,u)\overline{\sigma(\eta^{-1}z,\eta)}\big).
\]
Now applying part~(II) of Theorem~\ref{thm2} to $\{\varphi_x\}_{x\in \Lambda^\infty}$
with $\eta$ and $a=(u,1)$ we get
\begin{align*}
\psi_{x}(W_u)&=\varphi_{x}\big((u,1)\cdot\delta_{u}\big)\\
&=\varphi_{r(\eta)}\Big(\big(\eta u \eta^{-1},\sigma(\eta u,\eta^{-1})\sigma(\eta,u)\overline{\sigma(\eta^{-1},\eta)}\big) \cdot \delta_{\eta u \eta^{-1}}\Big)\\
&=\sigma(\eta u,\eta^{-1})\sigma(\eta,u)\overline{\sigma(\eta^{-1},\eta)}\varphi_{r(\eta)}\big((\eta u \eta^{-1},1)
 \cdot \delta_{\eta u \eta^{-1}}\big)\\
&=\sigma(\eta u,\eta^{-1})\sigma(\eta,u)\overline{\sigma(\eta^{-1},\eta)}\psi_{r(\eta)}\big(W_{\eta u \eta^{-1}}\big).
\end{align*}
To see \eqref{Nesh-formula-groupoid}, fix $f\in C_c(\GG,\sigma)$. Applying the formula
\eqref{Nesh-formula} for $\varphi$ we have
\begin{align}\label{state-groupois-bundle}
\psi(f)\notag&=\varphi(\Phi(f))=\int_{\GG^{(0)}}\sum_{u\in \Gx}\varphi_x(\Phi(f)(u)\cdot \delta_u)\, d\mu(x)\\
&=\int_{\GG^{(0)}}\sum_{u\in \Gx}f(u)\varphi_x((u,1)\cdot \delta_u)\, d\mu(x)
=\int_{\GG^{(0)}}\sum_{u\in \Gx}f(u)\psi_x(W_u)\, d\mu(x).
\end{align}
So the KMS$_\beta$ state $\psi$ yields a pair $\big(\mu,[\psi]_{\mu}\big)$ satisfying (I)
and (II), and $\psi$ in then given by \eqref{Nesh-formula-groupoid}.

For the converse, fix $\big(\mu,\{\psi_x\}_{x\in \GU}\big)$ satisfying (I) and (II).  Let
$\varphi_x=\psi_x\circ \Upsilon^{-1}$. For $g\in \PCG$ and $u\in \GU$, let $z_{g,u}\in \C$ be
the element such that $g(u)=(u, z_{g,u})$. The function $x\mapsto\sum_{u\in
\Gx}\varphi_x(g(u)\cdot \delta_u)=\sum_{u\in \Gx}z_{g,u}\psi_x(W_u)$ is $\mu$-measurable.
Therefore $\{\varphi_x\}_{x\in \GU}$ is a $\mu$-measurable field of states on
$C^*(\GG^x_x,\BB)$.  By (II) we have
\begin{align*}
\varphi_{x}((u,z)\cdot\delta_u)&=\psi_{x}\circ\Psi^{-1}\big((u,z)\cdot\delta_{u}\big)=\psi_x(zW_u)\\
&=z\sigma(\eta u,\eta^{-1})\sigma(\eta,u)\overline{\sigma(\eta^{-1},\eta)}\psi_{r(\eta)}\big(W_{\eta u\eta^{-1}}\big)\\
&=\psi_{r(\eta)}\big(z\sigma(\eta u,\eta^{-1})\sigma(\eta,u)\overline{\sigma(\eta^{-1},\eta)}W_{\eta u\eta^{-1}}\big)\\
&=\varphi_{r(\eta)}\big(\big(\eta u \eta^{-1},z\sigma(\eta u,\eta^{-1})\sigma(\eta,u)\overline{\sigma(\eta^{-1},\eta)}\big)\cdot \delta_{\eta u\eta^{-1}}\big)\\
&=\varphi_{r(\eta)}\big((\one_\eta (u,z) \one_\eta^*)\cdot \delta_{\eta u\eta^{-1}}\big).
\end{align*}
Thus $\big(\mu,\{\varphi_x\}_{x\in\GU}\big)$ is a pair as in Theorem~\ref{thm2}.
Therefore there is a KMS$_\beta$ state $\varphi:=\Theta(\mu,\Psi)$ on $C^*(\GG,\BB)$ satisfying
\eqref{Nesh-formula}. Now $\psi=\varphi\circ \Phi$ is a KMS$_\beta$ on $C^*(\GG,\sigma)$
and by \eqref{state-groupois-bundle} $\psi$ satisfies \eqref{Nesh-formula-groupoid}.
\end{proof}
\begin{remark}
Corollary~\ref{kms states for groupoid} applied to the trivial cocycle $\sigma\equiv 1$
recovers the results of Neshveyev in \cite[Theorem~1.3]{N}.
\end{remark}

\section{KMS states  on the $C^*$-algebras of twisted higher-rank graphs}\label{sec:KMS
k-graph}

\subsection{Higher-rank graphs}
Let $\Lambda$ be  a $k$-graph with vertex set $\Lambda^0$ and degree map
$d:\Lambda\rightarrow\N^k$ in the sense of  \cite{KP}. For any $n\in \N^k$, we write
$\Lambda^n:=\{\lambda\in \Lambda^*: d(\lambda)=n\}$. A $k$-graph $\Lambda$ is finite if
$\Lambda^n$ is finite for all $n\in \N^k$. Given $u,v\in \Lambda^0$,   $u\Lambda v$
denotes  $\{\lambda\in \Lambda: r(\lambda)=u \text { and } s(\lambda)=v\}$. We say
$\Lambda$ is \textit{strongly connected} if $u\Lambda v\neq\emptyset$ for every $u,v\in
\Lambda^0$. A $k$-graph $\Lambda$ has no sources if $u\Lambda^n\neq\emptyset$ for every
$u\in \Lambda^0$ and $n\in \N^k$ and it is row finite if $u\Lambda^n$ is finite for all
$u\in \Lambda^0$, and $n\in \N^k$.

A $\T$-valued 2-cocycle $c$ on $\Lambda$ is a map $c:\Lambda^2\rightarrow \T$ such that
$c(r(\lambda),\lambda)=c(\lambda,s(\lambda))=1$ for all $\lambda\in \Lambda$ and
$c(\lambda,\mu)c(\lambda\mu,\nu)=c(\mu,\nu)c(\lambda,\mu\nu)$ for all composable elements
$\lambda,\mu,\nu$. We write $Z^2(\Lambda,\T)$ for the group of all $\T$-valued 2-cocycles
on $\Lambda$.

Let  $\Omega_k:=\{(m,n)\in \N^k\times \N^k:m\leq n\}$. One can verify that that
$\Omega_k$ is a $k$-graph with  $r(m,n)=(m,m), s(m,n)=(n,n),$  $(m,n)(n,p)=(m,p)$ and
$d(m,n)=n-m$. We identify $\Omega_k^0$ with $\N^k$ by $(m,m)\mapsto m$. The set
\[\Lambda^\infty:=\{x:\Omega_k\rightarrow \Lambda: \text{ $x$ is a functor that intertwines  the degree maps}\}\]
is called the \textit{infinite-path space}  of $\Lambda$. For $l\in \N^k$, the shift map
$\rho^l:\Lambda^\infty\rightarrow \Lambda^\infty $ is given by $\rho^l(x)(m,n)=
x(m+l,n+l)$ for all $x\in \Lambda^\infty$ and $(m,n)\in \Omega_k$.

Let $\Lambda$ be a strongly connected finite $k$-graph. The set
\[\Per:=\{m-n: m,n\in \N^k, \rho^m(x)=\rho^n(x) \text{ for all } x\in \Lambda^\infty\}\subseteq \Z^k\]
is subgroup of $\Z^k$  and is called \textit{periodicity group} of $\Lambda$ (see
\cite[Proposition~5.2]{aHLRS}).

\subsection{The path  groupoid}
Suppose that  $\Lambda$ is a row finite $k$-graph with no sources. The set
\[\GL:=\{(x,l,y)\in \Lambda^\infty\times \Z^k\times\Lambda^\infty:l=m-n,m,n\in \N^k \text{ and } \sigma^m(z)=\sigma^{n}(z)\}\]
  is a groupoid with $(\GL)^{(0)}=\{(x,0,x):x\in \Lambda^\infty\}$ identified with $\Lambda^\infty$,  structure maps $r(x,l,y)=x$, $s(x,l,y)=y$, $(x,l,y)(y,l',z)=(x,l+l',z)$ and $(x,l,y)^{-1}=(y,-l,x)$. This groupoid is called \textit{infinite-path groupoid}.
	For $\lambda,\mu\in \Lambda$ with $s(\lambda)=r(\mu)$ let
\[Z(\lambda,\mu):=\{ (\lambda x,d(\lambda)-d(\mu),\mu x)\in \GL: x\in \Lambda^\infty \text{ and } r(x)=s(\lambda) \}.\]
The sets $\{Z(\lambda,\mu): \lambda,\mu\in \Lambda\}$ form  a basis for a locally compact
Hausdorff topology on $\GL$ in which it is an \'{e}tale groupoid (see
\cite[Proposition~2.8]{KP}).

Let $\Lambda\mathbin{_s*_s}\Lambda:=\{(\mu,\nu)\in \Lambda\times\Lambda:
s(\mu)=s(\nu)\}$. Let  $\PP$ be a subset of $\Lambda\mathbin{_s*_s}\Lambda$ such that
\begin{equation}\label{partition-form}
(\mu,s(\mu))\in \PP \text{ for all }\mu\in \Lambda\text{ and }\GL=\bigsqcup_{(\mu,\nu)\in \PP}Z(\mu,\nu).
\end{equation}
There is always such a $\PP$, see \cite[Lemma~6.6]{KPS1}. For each $\alpha\in \GL$, we
write $(\mu_\alpha,\nu_\alpha)$ for the element of $\PP$ such that $\alpha\in
Z(\mu_\alpha,\nu_\alpha)$.  Let $\hat{d}:\GL\rightarrow \Z^k$ be the function defined by
$\hat{d}(x,n,y)=n$. Given a  2-cocycle $c$ on $\Lambda$,   \cite[Lemma~6.3]{KPS1} says
that for every composable pair $\alpha,\beta\in \GL$ there are $\lambda,\iota,\kappa\in
\Lambda$ and $y\in \Lambda^\infty$ such that
\[\nu_\alpha\lambda=\mu_\beta\iota,\quad \mu_\alpha\lambda=\mu_{\alpha\beta}\kappa,\quad \nu_\beta\iota=\nu_{\alpha\beta}\kappa, \quad \text{ and}\]
\[
\alpha=(\mu_\alpha\lambda y,\hat{d}(\alpha),\nu_\alpha\lambda y),\quad \beta=(\mu_\beta\iota y,\hat{d}(\beta),\nu_\beta\iota y)\quad \text{and }
\alpha\beta=(\mu_{\alpha\beta}\kappa y,\hat{d}(\alpha\beta),\nu_{\alpha\beta}\kappa y).
\]
Furthermore, the formula
\[
\sigma_c(\alpha,\beta)=c(\mu_\alpha,\lambda)\overline{c(\nu_\alpha,\iota)}c(\mu_\beta,\iota)\overline{c(\nu_\beta,\iota)}\overline{c(\mu_{\alpha\beta},\kappa)}c(\nu_{\alpha\beta},\kappa).
\]
is a  continuous  2-cocycle  on $\GL$ and  does not depend on the choice of
$\lambda,\iota,\kappa$. Theorem~6.5 of \cite{KPS1} shows that continuous 2-cocycles on
$\GL$ obtained from different partitions $\PP,\PP'$ are cohomologous.

Let $\Lambda$ be a  strongly connected finite $k$-graph and take  $c\in Z^2(\Lambda,\T)$.
Let $\PP\subseteq\Lambda\mathbin{_s*_s}\Lambda$ be  as in \eqref{partition-form}. For
each $x\in \Lambda^\infty$, define $\sigma_c^x:\Per\rightarrow \T$ by
$\sigma_c^x(p,q):=\sigma_c((x,p,x),(x,q,x))$. Clearly $\sigma_c^x\in Z^2(\Per,\T)$. By
\cite[Lemma~3.3]{KPS2} the cohomology class of  $\sigma_c^x$ is independent of $x$. So by
the argument of Section~\ref{sec:bicharacters} there is a  bicharacter
$\omega_c$ on $\Per$ that is cohomologous to $\sigma_c^x$ for all $x\in \Lambda^\infty$.

\subsection{KMS states of preferred dynamics}
Given a finite $k$-graph $\Lambda$ and for $1\leq i\leq k$, let   $A_i\in M_{\Lambda^{0}}$ be the matrix  with entries $A_i(u,v):=|u\Lambda^{e_i}v|$. Writing $\rho(A_i)$ for the spectral radius of $A_i$, define $D:\GL\rightarrow \R$ by $D(x,n,y)=\sum_{i=1}^k n_i\ln \rho(A_i)$.  The function  $D$ is
locally constant and therefore it is a continuous $\R$-valued $1$-cocycle on $\GL$.
Lemma~12.1 of \cite{aHLRS} shows that there a unique probability measure $M$ on
$\Lambda^\infty$ with Radon--Nycodym cocycle $e^{D}$. This measure is a Borel measure and
satisfies
\begin{equation}\label{k-graph-measure}
M\big(x\in \Lambda^\infty:\{x\}\times \Per \times \{x\}\neq \Gx\}\big)=0.
\end{equation}
Given $\sigma\in  Z^2(\GL,\mathbb{T})$,  $D$ induces a dynamics $\tau$ on
$C^*(\GL,\sigma)$ such that $\tau_t(f)(x,m,y)=e^{it D(x,m,y)}f(x,m,y)$. Following
\cite{aHLRS} we call this dynamics the \textit{preferred dynamics}.

\begin{cor}\label{kms states for graph}
Suppose that $\Lambda$ is a strongly connected finite $k$-graph.  Let $c\in
Z^2(\Lambda,\mathbb{T})$ and let  $\PP$ be as in \eqref{partition-form}.  Suppose that
$\omega_c\in Z^2(\Per,\T)$ is a bicharacter cohomologous to
$\sigma_c^x(p,q)=\sigma_c((x,p,x),(x,q,x))$ for all $x\in \Lambda^\infty$.   Let $\tau$
be the preferred dynamics on $C^*(\GL,\sigma_c)$. Let   $M$ be the measure described at
\eqref{k-graph-measure}. There is a bijection between the simplex of  KMS$_1$ states of
$\big(C^*(\GL,\sigma_c),\tau\big)$  and
 the set of  $M$-equivalence classes $[\psi]_M$ of tracial states $\{\psi_x\}_{x\in \Lambda^\infty}$ on $C^*(\Per,\omega_c)$ such that for all $W_p\in \Per$ and $\eta:=(y,m,x)\in(\GL)_x$, we have
 \begin{equation}\label{property II of twisted}
\psi_{x}(W_p)=\sigma_c\big(\eta,(x,p,x)\big)\sigma_c\big((y,m+p,x),\eta^{-1}\big)\overline{\sigma_c(\eta^{-1},\eta)}\psi_{y}(W_p).
\end{equation}

The state corresponding to the class $[\psi]_{M}$ satisfies
\[
\psi(f)=\int_{\GG^{(0)}}\sum_{p\in \Per}f(x,p,x)\psi_x(W_p)\, dM(x)\quad\text{for all }f\in C_c(\GG,\sigma).
\]
\end{cor}
\begin{proof}
Fix $x\in \Lambda^\infty$ such that $\{x\}\times \Per \times \{x\}=\Gx$. Let $\delta^1 b$
be the $2$-coboundary such that $\omega_c=\delta^1 b \sigma_c^x$. Composing the
isomorphism  $W_p\mapsto b(p)W_p$ of $C^*(\Per,\omega_c)$ onto $C^*(\Per,\sigma_c^x)$ and
the isomorphism  $W_p\mapsto W_{(x,p,x)}:C^*(\Per,\sigma_c^x)\rightarrow
C^*(\Gx,\sigma_c)$,  we obtain an isomorphism $\Phi:C^*(\Per,\omega_c)\rightarrow
C^*(\Gx,\sigma_c)$ such that
\[\Phi(W_p)= b(p) W_{(x,p,x)} \text{ for all } p\in \Per. \]

Since $M$ is the only probability measure on $\Lambda^\infty$ with Radon--Nykodym cocycle
$e^{D}$, by Corollary~\ref{kms states for groupoid} it suffices to show that there is a
bijection between the fields of  tracial states  on $C^*(\Per,\omega_c)$ satisfying
\eqref{property II of twisted} and the $M$-measurable fields of  tracial states  on
$C^*(\Gx,\sigma_c)$ satisfying Corollary~\ref{kms states for groupoid} (II).

Let $\{\varphi_x\}_{x\in \Lambda^\infty}$ be an $M$-measurable field of tracial states
on $C^*(\Gx,\sigma_c)$ satisfying Corollary~\ref{kms states for groupoid} (II). Then
clearly  $\{ \varphi_x\circ\Phi\}_{x\in \Lambda^\infty}$ is a field of tracial states  on
$C^*(\Per,\omega_c)$. Applying part (II) of Corollary~\ref{kms states for groupoid} with
$\eta$ and $u=(x,p,x)$  we get
\begin{align*}
(\varphi_x\circ \Phi)(W_p)&=\varphi_x \big(b(p)W_{(x,p,x)}\big)\\
&=b(p)\sigma_c\big((y,m+p,x),\eta^{-1}\big)\sigma_c(\eta,(x,p,x))\overline{\sigma_c(\eta^{-1},\eta)}\varphi_y(W_{(y,p,y)})\\
&=\sigma_c\big((y,m+p,x),(x,p,x)\big)\sigma_c\big(\eta,(x,p,x)\big)\overline{\sigma_c(\eta^{-1},\eta)}\varphi_x\circ \Phi(W_p).
\end{align*}

Conversely let $\{\psi_x\}_{x\in \Lambda^\infty}$ be a field of states  on $C^*(\Per,\omega_c)$
satisfying \eqref{property II of twisted}.  Since $M$ is a Borel  measure  on
$\Lambda^\infty$,  for all $f\in C_c(\GG,\sigma)$, the function
\[x\mapsto\sum_{u\in \Gx}f(u)(\psi_x\circ \Phi^{-1})(W_u)=\sum_{p\in \Per}f(x,p,x)\overline{b(p)}\psi_x(W_p)\]
is continuous and hence  is $M$-measurable.  Therefore $\{\psi_x\circ \Phi^{-1}\}_{x\in
\Lambda^\infty}$ is a $M$-measurable field of tracial states on $C^*(\Gx,\sigma_c)$.

 Now applying \eqref{property II of twisted} to $\{\psi_x\}_{x\in \Lambda^\infty}$ with $\eta$ and $W_p$  we have
\begin{align*}
(\psi_x\circ \Phi^{-1})(W_u)&=\psi_x \big(\overline{b(p)}W_{p}\big)\\
&=\overline{b(p)}\sigma_c\big((y,m+p,x),\eta^{-1}\big)\sigma_c(\eta,(x,p,x))\overline{\sigma_c(\eta^{-1},\eta)}\psi_y(W_{p})\\
&=\sigma_c\big((y,m+p,x),(x,p,x)\big)\sigma_c\big(\eta,(x,p,x)\big)\overline{\sigma_c(\eta^{-1},\eta)}(\psi_y\circ \Phi^{-1})(W_{u}).
\end{align*}
\end{proof}

\subsection{KMS states and the invariance}
Given a strongly connected finite $k$-graph $\Lambda$, let $\I_\Lambda$ be the interior
of the isotropy $\Iso(\GL)$ in $\GL$. Define $\HH_\Lambda:=\GL/\I_\Lambda$ and let
$\pi:\GL\rightarrow \HH_\Lambda$ be the quotient map. Let $c\in Z^2(\Lambda,\mathbb{T})$
and let  $\PP$ be as in \eqref{partition-form}.  Suppose that   $\omega_c\in
Z^2(\Per,\T)$ is a bicharacter cohomologous to
$\sigma_c^x(p,q)=\sigma_c((x,p,x),(x,q,x))$ for all $x\in \Lambda^\infty$.   By
\cite[Lemma~3.6]{KPS2} there is a continuous
 $\widehat{Z}_{\omega_c}$-valued 1-cocycle $\tilde{r}^\sigma$ on $\HL$ such that
\[\tilde{r}^\sigma_{\pi(\gamma)}(p)=\sigma\big(\gamma,(y,p,y)\big)\sigma\big((x,m+p,y)\gamma^{-1}\big)\overline{\sigma(\gamma^{-1},\gamma)}\]
for all  $\gamma=(x,m,y)\in \GL$  and $p\in Z_{\omega_c}$. This induces an action $B$ of
$\HL$ on $\Lambda^\infty\times \widehat{Z}_{\omega_c}$ such that
\[B_{\pi(\gamma)}(s(\gamma),\chi)=\big(r(\gamma),\tilde{r}^\sigma_{\pi(\gamma)}\cdot\chi\big)\text{ for all }
\gamma\in \HL\text{ and } \chi\in \widehat{Z}_{\omega_c}.\]

\begin{cor}\label{cor:invariance}
Suppose that $\Lambda$ is a strongly connected finite $k$-graph.  Let $c\in
Z^2(\Lambda,\mathbb{T})$ and let  $\PP$ be as in \eqref{partition-form}.  Let
$\omega_c\in Z^2(\Per,\T)$ be a bicharacter cohomologous to
$\sigma_c^x(p,q)=\sigma_c((x,p,x),(x,q,x))$ for all $x\in \Lambda^\infty$.      Let
$\tau$ be the preferred dynamics on $C^*(\GL,\sigma_c)$ and let $M$ be the measure of
\eqref{k-graph-measure}. Then there is a bijection between the simplex of the KMS$_1$
states of $(C^*(\GL,\sigma_c),\tau)$  and  the set of  $M$-equivalence classes $[\psi]_M$
of tracial states $\{\psi_x\}_{x\in \Lambda^\infty}$ on $C^*(Z_{\omega_c})\cong \widehat{Z}_{\omega_c}$ that are
invariant under the action $B$, in the sense that
\[
B_{\pi(\gamma)}\big(s(\gamma),\psi_{r(\gamma)}\big)=\big(r(\gamma),\psi_{s(\gamma)}\big)\text{ for all }
\gamma\in \HL.
\]
\end{cor}
\begin{proof}
This follows from Corollary~\ref{kms states for graph} and Lemma~\ref{lemma:traces}.
\end{proof}

\subsection{A question of uniqueness for KMS$_1$ states}
If $c=1$, the results of \cite{aHLRS} show that $C^*(\GL,\sigma_1)$ has unique KMS$_1$
state if and only if it is simple (see Theorem~11.1 and Section~12 in \cite{aHLRS}).
Corollary~4.8 of \cite{KPS2} shows that $C^*(\GL,\sigma_c)$ is simple if and only if the
action $B$ of $\HL$ on $\Lambda^\infty \times \widehat{Z}_{\omega_c}$ is minimal.  So it
is natural to ask whether minimality of the action $B$ characterises the presence of a
unique KMS$_1$ state for the preferred dynamics? We have not been able to answer this
question. The following brief comments describe the difficulty in doing so.

The key point in \cite{aHLRS} that demonstrates that KMS states are parameterised by
measures on the dual of the periodicity group of the graph is the observation that in the
absence of a twist, the centrality of the copy of $C^*(\Per)$ in $C^*(\Lambda)$
can be used to show that KMS states are completely determined by their values on this
subalgebra. This, combined with Neshveyev's theorems, shows that  the field of states
$\{\psi_x\}_{x\in \Lambda^\infty}$ corresponding to a KMS state $\psi$ is, up to measure zero, a constant field
(see \cite[pages 27--28]{aHLRS}). The corresponding calculation fails in the twisted
setting.

However, we are able to show that, whether or not $\HL$ acts minimally on $\Lambda^\infty
\times \widehat{Z}_{\omega_c}$, there is an injective map from the states of
$C^*(Z_{\omega_c})$ that are invariant for the action of $\HL$ on
$\widehat{Z}_{\omega_c}$ induced by the cocycle $\tilde{r}^\sigma$ to the KMS states of
the $C^*$-algebra. It follows in particular that the Haar state on $C^*(Z_{\omega_c})$
induces a KMS state as expected.

\begin{cor}\label{cor aHLSR}
Let $\phi$ be a state on $C^*(Z_{\omega_c})$ such that $\tilde{r}_{\pi(\gamma)}\cdot
\phi=\phi$ for all $\gamma\in \HL$. Then there is a KMS$_1$ state $\psi_\phi$ of
$(C^*(\GL,\sigma),\tau)$ such that
\[
\psi_\phi(f)=\int_{\GG^{(0)}}\sum_{p\in \Per}f(x,p,x)\phi(W_p)\, dM(x)\quad\text{for all }f\in C_c(\GG,\sigma).
\]
The map $\phi\mapsto \psi_\phi$ is injective. In particular, there is a KMS$_1$ state
$\psi_{\Tr}$ of $(C^*(\GL, \sigma), \tau)$ such that
\[
\psi_{\Tr}(f) = \int_{\GG^{(0)}} f(x,0,x)\, dM(x)\quad\text{for all }f\in C_c(\GG,\sigma).
\]
\end{cor}
\begin{proof}
For each $x\in \Lambda^\infty$ define.
\begin{align*}
\psi_x=\begin{cases}
\phi&\text{if $\{x\}\times \Per\times \{x\} = \Gx$}\\
	0&\text{if $\{x\}\times \Per\times \{x\} \neq \Gx$.}\\
	\end{cases}
\end{align*}
Then $\psi_\phi := \Theta(M, \{\psi_x\}_{x\in \Lambda^\infty})$ satisfies the desired formula. The first
statement, and injectivity of $\phi \mapsto \psi_\phi$ follows from
Corollary~\ref{cor:invariance}. The final statement follows from the first statement
applied to the Haar trace $\Tr$ on $C^*(Z_{\omega_c})$.
\end{proof}

\begin{remark}
Suppose that $\HL$ acts minimally on $\Lambda^\infty \times \widehat{Z}_{\omega_c}$. Then
in particular the induced action $\tilde{B}$ of $\HL$ on $\widehat{Z}_{\omega_c}$ is
minimal. So if $\phi$ is a state of $C^*(Z_{\omega_c})$ that is invariant for $\tilde{B}$
as in Corollary~\ref{cor aHLSR}, then continuity ensures that the associated measure is
invariant for translation in $Z_{\omega_c}$, so must be equal to the Haar measure. So to
prove that $\psi_{\Tr}$ is the unique KMS$_1$-state when $C^*(\Lambda, c)$ is simple, it
would suffice to show that the map $\phi \mapsto \psi_\phi$ of Corollary~\ref{cor aHLSR}
is surjective.

One approach to this would be to establish that if $\{\psi_x\}_{x\in \Lambda^\infty}$ is an $M$-measurable
field of tracial states on $C^*(Z_{\omega_c})$, then the state $\phi$ given by $\phi :=
\int_{\Lambda^\infty} \psi_x \,dM(x)$ is $\tilde{B}$-invariant and satisfies $\psi_\phi =
\Theta(M, \{\psi\}_{x\in\Lambda^\infty})$, but we have not been able to establish either.
\end{remark}



\begin{thebibliography}{99}
\bibitem{BC} J.-B. Bost and A. Connes, Hecke algebras, type III factors  and phase
    transitions with spontaneous symmetry breaking in number theory, \emph{Selecta Math.
    (New Series)} {\bf 1} (1995), 411--457.

\bibitem{BR} O. Bratteli and D.W. Robinson, \emph{Operator algebras and quantum
    statistical mechanics II}, second ed., Springer-Verlag, Berlin, 1997.

\bibitem{CarlsenLarsen} T.M. Carlsen and N.S. Larsen, Partial actions and
    {KMS} states on relative graph {$C^*$}-algebras, \emph{J. Funct. Anal.} \textbf{271} (2016), 2090--2132.

\bibitem{EFW} M. Enomoto, M. Fujii and Y. Watatani, K{MS} states
    for gauge action on {$O_{A}$}, \emph{Math. Japon.} \textbf{29} (1984), 607--619.

\bibitem{FD} J.M.G. Fell and R.S. Doran \emph{Representations of $*$-algebras, locally
    compact groups, and Banach $*$-algebraic bundles, Basic representation theory of
    groups and algebras, Pure and Applied Mathematics,} vol.~1, Academic Press Inc.,
    Boston, MA, 1988.

\bibitem{aHLRS} A. an Huef, M. Laca, I. Raeburn and A. Sims, KMS states on the
    $C^*$-algebra of a higher-rank graph and periodicity in the path space, \emph{J.
    Funct. Anal.} \textbf{268} (2015), 1840--1875.

\bibitem{Kakariadis} E.T.A. Kakariadis, K{MS} states on {P}imsner algebras
    associated with {$\rm C^*$}-dynamical systems, \emph{J. Funct. Anal.} \textbf{269} (2015),
    325--354.

\bibitem{Kum} A. Kumjian, On {$C\sp \ast$}-diagonals, \emph{Canad. J. Math.}
    \textbf{38} (1986), 969--1008.

\bibitem{KP} A. Kumjian and D. Pask,  Higher-rank graph $C^*$-algebras, \emph{New York
    J. Math.} {\bf 6} (2000), 1--20.

\bibitem{KR} A. Kumjian and J. Renault,  KMS states on $C^*$-algebras associated to
    expansive maps, \emph{Proc. Amer. Math. Soc.} {\bf 134} (2006), 2067--2078.

\bibitem{KPS1} A. Kumjian, D. Pask and A. Sims,  On twisted higher-rank graph
    $C^*$-algebras, \emph{Trans. Amer. Math. Soc.} {\bf 367} (2015), 5177--5216.

\bibitem{KPS2} A. Kumjian, D. Pask and A. Sims,  Simplicity of twisted $C^*$-algebras of
    higher-rank graphs and crossed products by quasifree actions, \emph{J. Noncommut.
    Geom.} {\bf 10} (2016), 515--549.

\bibitem{LLNSW} M. Laca, N.S. Larsen, S. Neshveyev, A. Sims and S.B.G. Webster,
    Von {N}eumann algebras of strongly connected higher-rank graphs, \emph{Math. Ann.} \textbf{363} (2015), 657--678.

\bibitem{MW} P.S. Muhly and D.P. Williams, Equivalence and disintegration theorems for
    Fell bundles and their $C^*$-algebras, \emph{Dissertationes Math.} \textbf{456}
    (2008), 1--57.

\bibitem{N}  S. Neshveyev,  KMS states on the $C^*$-algebras of non-principal groupoids,
    \emph{J. Operator Theory} \textbf{70} (2013), 513--530.

\bibitem{OPT} D. Olesen, G.K. Pedersen and M. Takesaki, Ergodic actions of compact
    abelian groups, \emph{J. Operator  Theory} \textbf{3} (1980), 237--269.

\bibitem{Re} J. Renault, \emph{A groupoid approach to $C^*$-algebras,} Lecture Notes in
    Mathematics, vol.~793, Springer-Verlag., New York, 1980.

\bibitem{tfb} I. Raeburn and D.P. Williams, \emph{Morita equivalence and
    continuous-trace $C^*$-algebras}, Mathematical Surveys and Monographs, vol.~60, Amer.
    Math. Soc., Providence, 1998.

\bibitem{SWW} A. Sims, B. Whitehead and M.F. Whittaker, Twisted $C^*$-algebras associated
    to finitely aligned higher-rank graphs, \emph{Documenta Math.} \textbf{19} (2014),
    831--866.

\bibitem{Thomsen} K. Thomsen, KMS weights on graph $C^*$-algebras, \emph{Adv.
    Math.} \textbf{309} (2017), 334--391.

\bibitem{W} D.P. Williams, \emph{Crossed products of $C^*$-algebras}, Mathematical
    Surveys and Monographs, vol.~134, Amer. Math. Soc., Providence, 2007.

\bibitem{Yang} D. Yang, Type~III von Neumann algebras associated with
    2-graphs, \emph{Bull. Lond. Math. Soc.} \textbf{44} (2012), 675--686.
\end{thebibliography}
\end{document}